\documentclass{amsart}
\usepackage[left=2cm,right=1cm,
    top=2cm,bottom=2cm,bindingoffset=0cm]{geometry}

\usepackage[T2A]{fontenc}
\usepackage[utf8]{inputenc}
\usepackage[english]{babel}
\usepackage{amssymb}
\usepackage{a4wide}
\usepackage{amsfonts}
\usepackage{enumitem}

\makeatletter
\def\@settitle{\begin{center}%
    \baselineskip14\p@\relax
    \bfseries
    \@title
  \end{center}%
}
\makeatother

\usepackage{todonotes} % for comments and "to do" notes
\usepackage{amsfonts, amsmath, amssymb, amsthm, amscd, hyperref}

\usepackage{hyperref}
\hypersetup{
  colorlinks   = true, %Colours links instead of ugly boxes
  urlcolor     = blue, %Colour for external hyperlinks
  linkcolor    = blue, %Colour of internal links
  citecolor   = red %Colour of citations
}

\usepackage{graphicx}

\newtheorem{theorem}{Theorem}

\newtheorem*{theorem*}{Theorem}
\newtheorem{proposition}[theorem]{Proposition}
\newtheorem*{proposition*}{Proposition}

\newtheorem*{remark*}{Remark}
\newtheorem{lemma}[theorem]{Lemma}
\newtheorem*{lemma*}{Lemma}

\newtheorem*{problem*}{Problem}

\newtheorem*{claim*}{Claim}
\newtheorem{observation}[theorem]{Observation}
\newtheorem*{observation*}{Observation}

\newtheorem*{conjecture*}{Conjecture}
\newtheorem{corollary}[theorem]{Corollary}
\newtheorem*{corollary*}{Corollary}
\newtheorem*{example*}{Example}
\newtheorem*{definitions*}{Обозначения}

\newtheorem*{definition*}{Определение}

\def\geq{\geqslant}
\def\leq{\leqslant}

\def\N{\mathbb{N}}
\def\Z{\mathbb{Z}}
\def\Q{\mathbb{Q}}
\def\R{\mathbb{R}}

\def\eps{\varepsilon}

\def\ra{\rightarrow}

\renewcommand{\notin}{\not\in}

\def\pr{\text{pr}}
\def\H{\mathcal{H}}
\def\T{\mathcal{T}}
\def\V{\mathcal{V}}

\usepackage{setspace}
\author{Arthur Bikeev$^{12}$}
\thanks{
$^1$ Email: bikeev99@mail.ru, ORCID: 0009-0008-6655-9575}
\thanks{
$^2$ Moscow Institute of Physics and Technology, Moscow, Russia.}

\begin{document}
\title{\Large Isomorphisms of unit distance graphs of layers$^*$.
\footnote{$*$ I am grateful for useful discussions to A. Raigorodskii and V. Voronov.}}
\Large

\begin{abstract}
For any $\varepsilon \in (0,+\infty)$, consider the metric spaces $\mathbb{R} \times [0,\varepsilon]$ in the Euclidean plane named layers or strips. B. Baslaugh in 1998 found the minimal width $\varepsilon \in (0,1)$ of a layer such that its unit distance graph contains a cycle of a given odd length $k$. 
The first of the main results of this paper is the fact that the unit distance graphs of two layers $\mathbb{R} \times [0,\varepsilon_1], \mathbb{R} \times [0,\varepsilon_2]$ are non-isomorphic for any different values $\varepsilon_1,\varepsilon_2 \in (0,+\infty)$.

We also get a multidimensional analogue of this theorem. For given $n,m \in \mathbb{N}, p \in (1,+\infty), \varepsilon \in (0,+\infty)$, we say that the metric space on $\mathbb{R}^n \times [0,\varepsilon]^m$ with the metric space distance generated by $l_p$-norm in $\mathbb{R}^{n+m}$ is a \textit{layer} $L(n,m,p,\varepsilon)$. We show that the unit distance graphs of layers $L(n,m,p,\varepsilon_1), L(n,m,p,\varepsilon_2)$ are non-isomorphic for $\varepsilon_1 \neq \varepsilon_2$.

The third main result of this paper is the theorem that, for $n \geq 2, \varepsilon > 0$, any automorphism $\phi$ of the unit distance graph of layer $L = L(n,1,2,\varepsilon) = \mathbb{R}^n \times [0,\varepsilon]$ is an isometry. This is related to the Beckman-Quarles theorem of 1953, which states that any unit-preserving mapping of $\mathbb{R}^n$ is an isometry, and to the rational analogue of this theorem obtained by A. Sokolov in 2023.

\end{abstract}

\maketitle

\tableofcontents

\section{Introduction}\label{sec:intro}
The \textit{unit distance graph} of a set $V$ in some metric space is the graph in which $V$ is the set of vertices and two vertices are connected by edge if and only if the distance between them is $1$.

 The well-known Hadwiger–Nelson problem, named after H. Hadwiger and E. Nelson, asks for the \textit{chromatic number} $\chi(\R^2)$ of the Euclidean plane, i. e. the minimal number of colors required to color the unit distance graph of the plane such that no two vertices connected by an edge have the same color. The first non-trivial lower bound $\chi(\R^2) \geq 4$ was obtained in 1961 by brothers L. and W. Moser, and the upper bound $\chi(\R^2) \leq 7$ was obtained by H. Hadwiger in 1961. Since the breakthrough result of A. D. N. J. de Grey \cite{Gr18} in 2018 it is known that at least $5$ colors are necessary. 
 So, it is now known that $$5 \leq \chi(\R^2) \leq 7.$$  See \cite{Soi08} for more details.
 %Since its inception in 1950 it is known that seven colors suffice, and since the breakthrough result of de Grey \cite{Gr18} in 2018 it is known that at least $5$ colors are necessary.
 An analogous question can be posed for the chromatic number of spaces $\R^n$, $\Q^n$ with the standard Euclidean metric. In particular, it is known that $$(1.239...+0(1))^n \leq \chi(\R^n) \leq (3+o(1))^n.$$ The lower bound belongs to A. M. Raigorodskii (\cite{Ra00}). The upper bound was obtained by D. Larman and A. Rogers in \cite{LR72} and re-proved by R. Prosanov in \cite{Pro20}. See, for example, \cite{KVC}, \cite{VKSC} for more details on the known results for $\R^3, \Q^n$. In addition, there exist a number of results for the metric space on $\R^n$ with $l_p$-metric defined as
 $$||x||_p = \left(x_1^p + x_2 ^p + \dots + x_n^p \right)^{\frac{1}{p}}.$$

 See Raigorodskii's surveys \cite{Ra01}, \cite{Ra12} for results on the chromatic numbers in various metric spaces.

 For given $n,m \in \N, p \in (1,+\infty), \eps \in (0,+\infty)$, we say that the metric space on $\R^n \times [0,\eps]^m$ with the metric space distance generated by $l_p$-norm in $\R^{n+m}$ is a \textit{layer} $L(n,m,p,\eps)$. The layer $L(1,1,2,\eps)$ (that is, the metric space on $\R\times[0,\eps]$ with the standard Euclidean metric) is called a \textit{strip}.

 The chromatic numbers of the layers actively began to be studied by A. Kanel-Belov, V. Voronov and D. Cherkashin in \cite{KVC}. The authors obtained the following estimates: 
 $$5 \leq \chi(L(2,1,2,\eps)) = \chi(\R^2 \times [0,\eps]) = \leq 7 \text{ for } 0<\eps < \sqrt{3/7},$$
 $$\chi(L(2,d,2,\eps))= \chi(\R^2 \times [0,\eps]^d) \leq 7\text{ for }0<\eps < \eps_0(d),$$
 $$\chi(L(2,2,2,\eps))= \chi(\R^2 \times [0,\eps]^2) \geq 6 \text{ for } \eps >0,$$ and a number of results for rational layers. In \cite{VKSC} A. Kanel-Belov, V. Voronov, G. Strukov and D. Cherkashin give the estimate $$10 \leq \chi(L(3,6,2,\eps))= \chi(\R^3 \times [0,\eps]^6) \leq 15.$$
 In \cite{Sh19} L. Shabanov proved a Tur\'an-type lower bound on the minimal number of edges in a finite subgraph of the unit distance graph of the layer $L(2,d,2,\eps) = \R^2 \times [0,\eps]^d$, using its independence number.

Some other results for infinite and finite planar strips were obtained in \cite{Po21}, \cite{ACLMSS}, and \cite{OMH20}.
Also, there are a number of papers related to coloring of strips with specific restrictions. For example, B. Bauslaugh in \cite{Bo19} obtained some estimates for colorings with
forbidden intervals of distances between one-color points. V. Kirova in \cite{Ki23} deals with forbidden one-color arithmetic progressions.  N. Alon, M. Buci\'c and L. Sauermann in \cite{ABS23} consider the question of possible number of edges in finite subgraphs of unit distance graphs for various normed spaces.

B. Bauslaugh in \cite{Ba98} found the minimal width $\eps \in (0,1)$ of a strip such that its unit distance graph contains a cycle of a given odd length $k$.
%Baslaugh in \cite{Ba98} finds the minimal width $\eps$ of the planar layer $\R \times [0, \eps]$ such that its unit distance graph contains a cycle of given odd length $k$.
In particular, this result gives the following estimate for the chromatic number of the strip $\R \times [0, \eps]$: 
$$\chi(\R \times [0,\eps]) \leq 3\text{ if and only if }\eps \leq \sqrt{3}/2.$$

The first of the main results of this paper is the fact that the unit distance graphs of two layers $\R \times [0,\eps_1], \R \times [0,\eps_2]$ are non-isomorphic for any different values $\eps_1,\eps_2 \in (0,+\infty)$.

We also get a multidimensional analogue of this theorem. In the second main result, we show that the unit distance graphs of layers $L(n,m,p,\eps_1), L(n,m,p,\eps_2)$ are non-isomorphic for $\eps_1 \neq \eps_2$.

There are a number of related results. 

L. Lichev and T. Mihaylov in \cite{LM24} consider a graph embeddable in $\R^d$ so that two vertices $u$ and $v$ form an edge if and only if their images in the embedding are at a distance in the interval $[R_1,R_2]$, instead of standard unit-distance graphs. They showed that the family $\mathcal{A}_d(R_1, R_2)$ of such graphs is uniquely characterized by $R_1/R_2$.

F. S. Beckman and D. A. Quarles in \cite{BQ53} proved that any unit-preserving mapping of $\R^n$ is an isometry. A. Sokolov in \cite{So23} obtained a rational analogue of Beckman--Quarles theorem. 

In \cite{Al70}, A. V. Aleksandrov posed the problem of characterising those
at least two, but finite dimensional normed spaces $X$ such that any transformation $\phi:X\ra X$ which preserves distance $1$ is an isometry,  hence it is usually
called the \textit{Aleksandrov (conservative distance) problem}. Some modified version of this problem were solved, but in full generality it is still open. For more details and related results, see, for example, \cite{Ku86}, \cite{Ra07}, \cite{Ge17}, \cite{HT18}.

The third result of this paper is theorem that, for $n \geq 2, \eps > 0$, any automorphism $\phi$ of the unit distance graph of layer $L = L(n,1,2,\eps) = \R^n \times [0,\eps]$ is an isometry.

\section{Main results}

\begin{theorem}\label{t:planar}
    For given $\eps_1,\eps_2 \in (0,+\infty)$, the unit distance graphs of the strips $L(1,1,2,\eps_1)$ $= \R \times [0,\eps_1]$ and $L(1,1,2,\eps_2) = \R \times [0,\eps_2]$ are isomorphic if and only if $\eps_1=\eps_2$.
\end{theorem}

The main idea of the proof is to show that any isomorphism of the unit distance graphs of layers must preserve some geometric properties. For example, any such isomorphism maps a segment parallel to $[0,\eps_1]$ into a segment parallel to $[0,\eps_2]$ and preserves the order of points in these segments. Next, we show that this isomorphism preserves sets named $(N,M)$-combs (see Fig. \ref{fig:comb}). A similar argument can be found in \cite{HKOSY} by T. Hayashi, A. Kawamura, Y. Otachi, H. Shinohara and K. Yamazaki.
Since the ratio $N /M$ gives a lower bound on the minimal width of a layer containing a $(N,M)$-comb, and this lower bound grows monotonically depending on the parameter $N/M$, we can distinguish the unit distance graphs of layers of different widths. In the case $\eps_i \geq 1$, we use some additional arguments.

Next, we introduce a multidimensional generalization of Theorem \ref{t:planar}.

\begin{theorem}\label{t:main}
For given $n \in \N \cap [2,+\infty)$, $m\in \N$, $p\in (1,+\infty)$, $\eps_1,\eps_2 \in (0,+\infty)$, the unit distance graphs of the layers $L_1 = L(n,m,p,\eps_1), L_2 = L(n,m,p,\eps_2)$ are isomorphic if and only if $\eps_1 = \eps_2$.
\end{theorem}

\begin{remark*}
    The condition $p\in (1,+\infty)$ in the formulation of Theorem \ref{t:main} is necessary because the proof is true only for smooth norms. The unit distance graphs of the layers $L(n,1,+\infty, \eps_1)$ and $L(n,1,+\infty, \eps_2)$ are obviously isomorphic for any $\eps_1, \eps_2 \in (0,1)$.
\end{remark*}

In the proof of Theorem \ref{t:main}, we show that any isomorphism of the unit distance graphs of the layers $L_1, L_2$ preserves straight lines (parallel to the subspace $\R^n$) and preserves the distances on these lines.
    The key difference between the cases $n=1$ and $n\geq 2$ can be explained by Proposition \ref{prop:dist2} that does not hold in the case $n=1$. In fact, this makes the proof in the case $n \geq 2$ a little easier.

    In fact, the proof of Theorem \ref{t:main} with some simple additional details gives the following theorem related to the Aleksandrov problem.

    \begin{theorem}\label{t:appendix}
        For $n \geq 2, \eps > 0$, any automorphism $\phi$ of the unit distance graph of layer $L = L(n,1,2,\eps) = \R^n \times [0,\eps]$ is an isometry.
    \end{theorem}

\section{Proof of Theorem \ref{t:planar}}\label{sec:proof_planar}

Denote by $||\cdot||_2$ the standard Euclidean norm.

Let $\eps_1, \eps_2 \in (0, +\infty)$. Consider the layers $L_1 = L(1,1,2,\eps_1) = \R\times [0,\eps_1]$ and $L_2 = L(1,1,2,\eps_2) = \R\times [0,\eps_2]$, and denote the unit distance graphs of these layers by $G_1$ and $G_2$, respectively. Denote by $\rho_i(x,y)$ the length of shortest path between two vertices $x,y$ in the graph $G_i$.

Let us suppose that there exists an isomorphism $f: L_1 \ra L_2$ of the graphs $G_1, G_2$. We will prove that this implies $\eps_1 = \eps_2$.

\begin{observation}
    For any $x,y \in L_1$, we have $\rho_1(x,y) = \rho_2(f(x),f(y))$, because $f$ is an isomorphism of the graphs $G_1,G_2$. In particular, $x=y$ if and only if $f(x)=f(y)$.
    
    For any $x,y \in L_i$, we have $\rho_i(x,y) = 1$ if and only if $||x-y||_2=1$. Therefore, $||x-y||_2 = 1$ if and only if $||f(x)-f(y)||_2 = 1$.
\end{observation}

For $i \in \{1,2\}$, denote by $\pr(x)$ the projection of a point $x \in L_i$ onto the line $\R \times \{0\} \subset L_i$. We say that any line (segment, etc.) in $L_i$ parallel to $\R\times \{0\}$ is a \textit{horizontal} line (segment, etc.), and any segment in $L_i$ parallel to the segment $\{0\} \times [0,\eps_i]$ is a \textit{vertical}  segment (that is, a segment $[x,y]$ is vertical if $\pr(x) = \pr(y)$).

Denote by $\left\lceil \alpha \right\rceil$ the ceiling function, that is, the least integer greater than or equal to $\alpha$.

\begin{proposition}\label{prop:eq}
     There exists a constant $c=c(\eps_i) >0$ such that the following holds. For any $a,b \in L_i$, if $||a-b||_2 > c$, then $\rho_i(a,b) = \left\lceil ||a-b||_2\right\rceil$. In particular, $G_i$ is connected.
\end{proposition}

\begin{proof}
If $||a-b||_2 \in \N$, then the statement is obvious.
Let
$$a,b \in L_i, ||a-b||_2 = n+\delta, \delta \in (0,1), n >10\eps+10.$$ 

Without loss of generality, let
$$a =(x_1,y_1), b=(x_2,y_2), x_1<x_2, y_1 \leq y_2.$$

Take the points 
$$a' = a+\left\lceil \frac{n}{3}\right\rceil \cdot \frac{b-a}{||a-b||_2},$$
$$b' = b+\left\lceil \frac{n}{3}\right\rceil \cdot \frac{a-b}{||a-b||_2}$$ on the segment $[a,b]$ (they forms the "middle part" of the segment). 

We have 
$$\rho_{i}(a,a') = ||a-a'||_2 = \left\lceil \frac{n}{3}\right\rceil, \rho_{i}(b,b') = ||b-b'||_2 = \left\lceil \frac{n}{3}\right\rceil.$$
Then 
$$y_1,y_2 > \frac{\eps_i}{4} \text{ or } y_1,y_2 < \frac{3\eps_i}{4},$$  
that is, the distances from these points to at least on of the lines $\R\times\{0\}, \R\times\{\eps_i\}$ are sufficiently large.

For definiteness, let $y_1,y_2 > \frac{\eps_i}{4}$.
Take the point $b'' = b' + \frac{a-b}{||a-b||_2}$ on the distance $1$ from $b'$. Then $$||a'-b''||_2 = n- 2\cdot\left\lceil \frac{n}{3}\right\rceil -1+ \delta.$$

We are left to prove that 
$$\rho_{G_i}(a',b'') = \left\lceil ||a'-b''||_2\right\rceil= n- 2\cdot\left\lceil \frac{n}{3}\right\rceil,$$ for any sufficiently large $n$.

Since 
$$\angle a'b''\pr(b') > \pi/4,\;||a'-\pr(a')||_2> \frac{\eps_i}{4},\;||b'-b''||_2 = 1,$$ there exists a rectangle $a'b''cd$ containing in $L_i$ such that $||a'-d||_2 > \min\left(1,\frac{\eps_i}{4}\right)$, see Fig.\ref{fig:prop_C_planar}. Then, for any sufficiently large $n$, there exists a path of length $\left\lceil ||a'-b''||_2\right\rceil$ from $a'$ to $b''$ in this rectangle, by a construction shown on Fig. \ref{fig:polygonal_chain}. In this construction, we use the inequality $||a'-d||_2 > \min\left(1,\frac{\eps_i}{4}\right)$, that allows us to choose $$||a'-a'_1||_2 \in \left[2\sqrt{1-\min\left(1,\frac{\eps_i}{4}\right)^2},2\right].$$

Therefore,
$$\rho_{i} (a',b'') = n- 2\cdot\left\lceil \frac{n}{3}\right\rceil;$$
$$\rho_{i} (a,b) \leq \rho_{i}(a,a') + \rho_{i}(a',b'')+1+\rho_{i}(b',b) \leq n+1 = \left\lceil ||a-b||_2\right\rceil.$$

\end{proof}
\begin{figure}[ht]
\center{\includegraphics[scale=0.5, width=310pt]{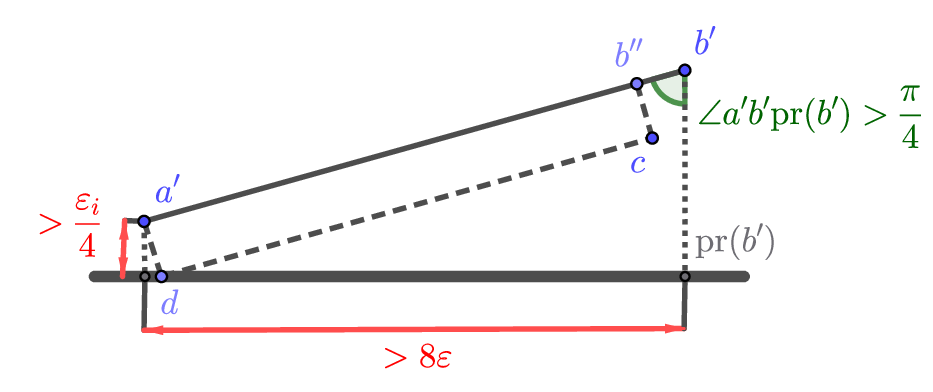}}
\caption{The rectangle $a'b''cd$}
\label{fig:prop_C_planar}
\end{figure}

\begin{figure}[ht]
\center{\includegraphics[scale=0.5, width=320pt]{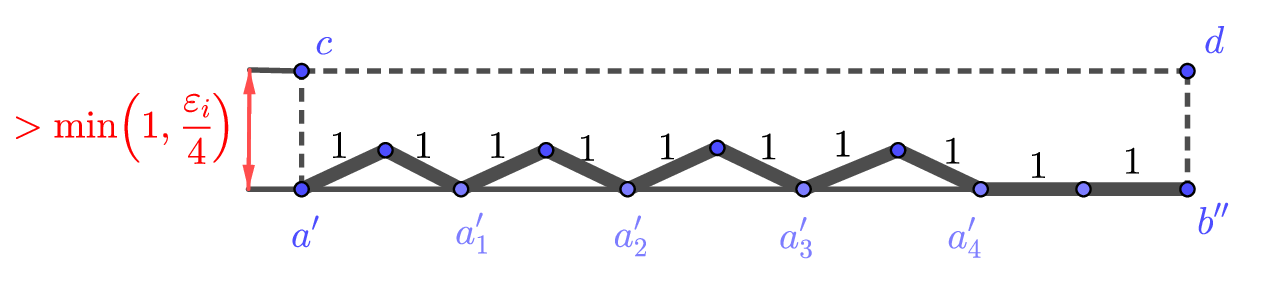}}
\caption{The shortest path from $a'$ to $b''$ in $G_i$.}
\label{fig:polygonal_chain}
\end{figure}

\begin{corollary}\label{cor:inf}
   Let $y \in L_i$ be a point, and let $\{z_m: m\in N\} \subset L_i$ be a sequence of points. Then the condition
   $$\rho_{i}(z_m, y) \xrightarrow{m\rightarrow \infty} + \infty$$   
 is equivalent to
    $$||z_m-y||_2 \xrightarrow{m\rightarrow \infty} + \infty.$$
\end{corollary}

Now we are ready to prove a key auxiliary lemma.

\begin{lemma}\label{l:main1}
Let $i \in \{1,2\}$, and let $x',x'',y$ be distinct points in $L_i$. Then the following two conditions are equivalent.

$(\Gamma)$. For any sequence of points $z_1,z_2,... \in L_i$, such that $\rho_{i}(z_m, y) \xrightarrow{m\rightarrow \infty} + \infty$, there exists $m_0 \in \N$, such that for any $m > m_0$ we have $\min(\rho_{i}(z_m, x')$, $\rho_{i}(z_m, x''))$ $<\rho_{i}(y-z_m)$.
    
$(\Omega)$. The points $\pr(x'),\pr(x'')$ are distinct. The projection $\pr(y)$ lies strictly in the interior of the segment $[\pr(x'),\pr(x'')]$.
\end{lemma}

\begin{proof}
By a slight abuse of the notation, we compare points on the line $\R\times \{0\}$ as real numbers.

    \textit{Proof of the implication $(\neg\Omega)\Rightarrow(\neg\Gamma)$}. Let condition $(\Omega)$ does not hold.
    For definiteness, let $\pr(x') \leq \pr(y), \pr(x'') \leq \pr(y)$. Denote ${\textbf u}=(1,0)$ and consider the points $z_m = y + m{\textbf u}$ for $m \in \N$. Then for any $m \in \N, x \in \{x',x''\}$ we have
$$m = ||y-z_m||_2 = ||\pr(y)-\pr(z_m)||_2 = \rho_{i}(\pr(z_m), \pr(y)) \leq \rho_{i}(\pr(z_m), \pr(x)).$$
Hence, condition $(\Gamma)$ does not hold (see Fig. \ref{fig:sup_lemma_planar_1}). 

    \textit{Proof of the implication $(\Omega)\Rightarrow(\Gamma)$}.
    Let condition $(\Omega)$ holds. For definiteness, let $\pr(x') < \pr(y) < \pr(x'')$.
    Consider an arbitrary sequence $z_1,z_2,... \in L_i$, such that $\rho_{i} (z_m, y)\xrightarrow{m\rightarrow \infty} + \infty$. By Corollary \ref{cor:inf} we have $||z_m-y||_2 \xrightarrow{m\rightarrow \infty} + \infty$. Then $z_n$ contains a subsequence $\{z_{k_m}: m\in \N\}$ such that $\pr(z_{k_m}) \xrightarrow{m\rightarrow \infty} -\infty$ or $\pr(z_{k_m}) \xrightarrow{m\rightarrow \infty} +\infty$. Without loss of generality, we assume that $z_m = z_{k_m}$, and $\pr(z_m) \xrightarrow{m\rightarrow \infty} +\infty$. Then, for any sufficiently large $m$, the angle $\angle yx''z_m$ is obtuse, and $||y-z_m||_2 > ||x''-z_m||_2$.  Thus, condition $(\Gamma)$ holds (see Fig. \ref{fig:sup_lemma_planar_2}).
\end{proof}

\begin{figure}[ht]
\center{\includegraphics[scale=0.5, width=250pt]{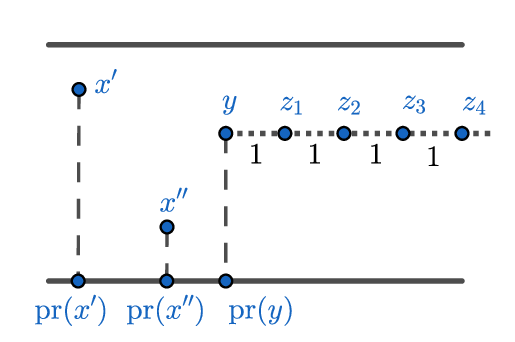}}
\caption{Proof of the implication $(\neg\Omega)\Leftarrow(\neg\Gamma)$. The point $y$ is closer to $z_m$ than $x'$ and $x''$.}
\label{fig:sup_lemma_planar_1}
\end{figure}

\begin{figure}[ht]
\center{\includegraphics[scale=0.5, width=250pt]{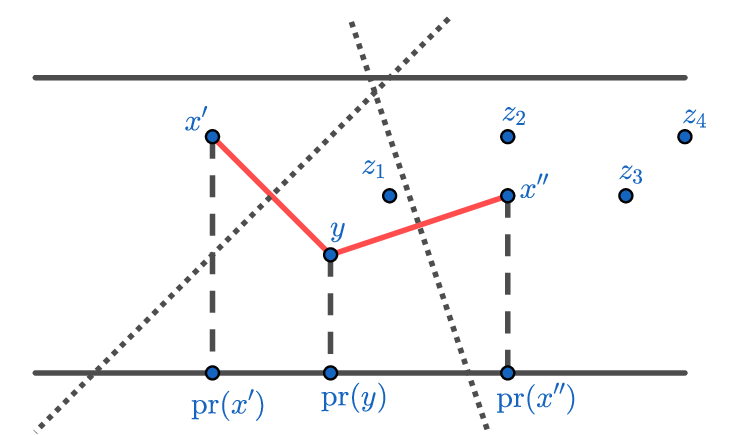}}
\caption{Proof of the implication $(\Omega)\Leftarrow(\Gamma)$. The perpendicular bisectors to the segments $[x',y]$ and $[x'',y]$ separate $z_m$ from $y$ for sufficiently large $m$.}
\label{fig:sup_lemma_planar_2}
\end{figure}

The following lemma shows that $f$ preserves the vertical segments.

\begin{lemma}\label{l:vert}
Let $x',x'' \in L_1$ be two distinct points. Then $\pr(x') = \pr(x'')$ if and only if $\pr(f(x')) = \pr(f(x''))$.
\end{lemma}

\begin{proof}

The inequality $\pr(x') \neq \pr(x'')$ holds if and only if there exists a point $y \in L_1$ such that condition $(\Omega)$ of Lemma \ref{l:main1} applied to $x',x'',y$ is true. By Lemma \ref{l:main1}, this holds if and only if there exists a point $y \in L_1$ such that condition $(\Gamma)$ of Lemma \ref{l:main1} applied to $x',x'',y$ is true.

Analogously, the inequality $\pr(f(x')) \neq \pr(f(x''))$ is equivalent to the following: there exists a point $\widetilde y \in L_2$ such that condition $(\Gamma)$ applied to $f(x'),f(x''),\widetilde y \in L_2$ is true.

Since $f$ is an isomorphism of graphs, for any $a,b \in L_1$ we have $\rho_{1}(a,b) = \rho_{2}(f(a), f(b))$. Thus, condition $(\Gamma)$ for $x',x'',y$ is equivalent to condition $(\Gamma)$ for $f(x'),f(x''),f(y)$ for any $x',x'',y \in L_1$.

Thus, $\pr(x') = \pr(x'')$ if and only if $\pr(f(x')) = \pr(f(x''))$.
\end{proof}

The following lemma shows that $f$ preserves the order of points on a vertical segment.

\begin{lemma}\label{l:vert_xyz}
    Let $x,y,z \in L_1$ be distinct points lying on the same vertical line, and let $y \in (x,z)$. Then $f(x),f(y),f(z) \in L_2$ are distinct points lying on the same vertical line, and $f(y) \in (f(x), f(z))$.
\end{lemma}

\begin{proof}
The points $f(x),f(y),f(z)$ are distinct since $f$ is an isomorphism of graphs.
    Thus, by Lemma \ref{l:vert}, $f(x),f(y),f(z) \in L_2$ are distinct points on the same vertical line because $\pr(x)=\pr(y)=\pr(z) \Leftrightarrow \pr(f(x)) = \pr(f(y))=\pr(f(z))$.

    The perpendicular bisectors to the segments $[x,y]$ and $[y,z]$ are two parallel horizontal lines. Thus, for any point $a \in L_1$ we have $\max(||a-x||_2, ||a-z||_2) > ||a-y||_2$, see Fig. \ref{fig:y_between_xz}.a). Then, due to Proposition \ref{prop:eq}, there exists a constant $C$ such that $\rho_1(a,y) > C$ implies $\max(\rho_{1}(a,x), \rho_{1}(a,z)) \geq \rho_{1}(a,y)$ for any $a \in L_1$.
    
    Then, since $f$ is an isomorphism of graphs, $\rho_1(a,y)>C$ implies $$\max(\rho_{2}(f(a),f(x)), \rho_{2}(f(a),f(z))) \geq \rho_{2}(f(a),f(y))$$ for any $a \in L_1$. Due to Proposition \ref{prop:eq}, there exists a constant $C'$ such that $\rho_2(b,f(y)>C'$ implies $\rho_2(b,f(y)) > \rho_1(f^{-1}(a),y)$ for any $b \in L_2$.
    
    Hence, $\max(\rho_{2}(b,f(x)), \rho_{2}(b,f(z))) \geq \rho_{2}(b,f(y))$ for any $b \in L_2$ that satisfies $\rho_2(b,f(y)>C'$. However, if $f(y) \notin (f(x), f(z))$, then this is not true. It is enough to consider an isosceles triangle $bf(x)f(z)$ that satisfies $||b-f(x)||_2=||b-f(z)||_2 = m \in \N \cap (C'+1, +\infty)$, see Fig. \ref{fig:y_between_xz}.b).

    Thus, $f(y) \in (f(x),f(z))$.
\end{proof}

\begin{figure}[ht]
\center{\includegraphics[scale=0.5, width=350pt]{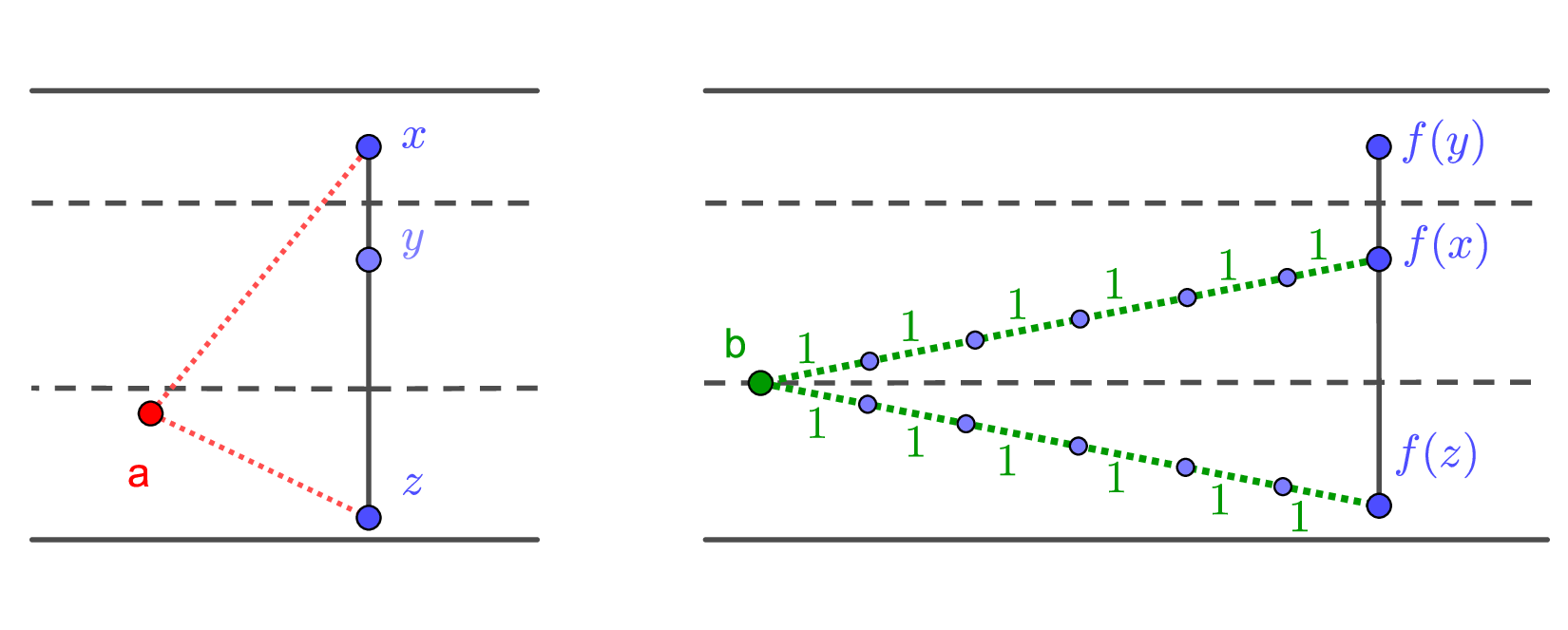}}
\caption{a). The case $y \in [x,z]$, $||a-x||_2 \geq ||a-y||_2$.  b) The case $f(x) \in [f(y), f(z)]$, $\max(\rho_{1}(b,f(x)), \rho_{1}(b,f(z))) < \rho_{1}(b,f(y))$}
\label{fig:y_between_xz}
\end{figure}

\begin{corollary}\label{cor:boundary}
    For any $y \in L_1$ such that $y \in \R\times\{0, \eps_1 \}$, we have $f(y) \in \R\times\{0, \eps_2 \}$.
\end{corollary}

\begin{proof}
    We have $y \in \R\times\{0, \eps_1 \}$ if and only if there are no such points $x,z \in L_1$, such that $x,y,z$ are distinct points on the same vertical line, and $y\in [x,z]$. At the same time, $f(y) \in \R\times\{0, \eps_2 \}$ if and only if there are no such points $x',z'\in L_2$, such that $x',f(y),z'$ are distinct points on the same vertical line, and $f(y)\in [x',z']$. We are only left to apply Lemma \ref{l:vert_xyz}.
\end{proof}

The following lemma shows that $f$ preserves the horizontal segments of integer length.

\begin{lemma}\label{l:horizontal}
    Let $x,y \in L_1$, and $||x-y||_2=1$. Then $x,y$ lie on the same horizontal line if and only if $f(x), f(y)$ lie on the same horizontal line.
\end{lemma}

\begin{proof}
    Note that $\rho_{1}(\pr(x),\pr(y)) = 1$ if and only if the segment $[x,y]$ is horizontal. On the other hand, $\rho_{1}(\pr(x),\pr(y)) = 1$ if and only if there are no such points $x',y'$ that $\rho_{1}(x',y')=1$ and the segment $[\pr(x'),\pr(y')]$ contains the points $\pr(x),\pr(y)$ strictly in the interior, see Fig. \ref{fig:horizontal}.

Therefore, $[x,y]$ is horizontal if and only if there are no such points $x',y'$ that $\rho_{1}(x',y')=1$ and the segment $[\pr(x'),\pr(y')]$ contains the points $\pr(x),\pr(y)$ strictly in the interior. The analogous statement also holds for $f(x),f(y)$.
    
    Since $f$ is an isomorphism of graphs, we have $\rho_{1}(x',y')=1$ if and only if $\rho_{1}(f(x'),f(y'))=1$. Due to Lemma \ref{l:main1} applied to $x',x,y'$ and applied to $x',y,y'$, the points $\pr(x), \pr(y)$ lie strictly in the interior of the segment $[\pr(x'),\pr(y')]$ if and only if the points $\pr(f(x)),\pr(f(y))$ lie strictly in the interior of the segment $[\pr(f(x')),\pr(f(y'))]$.
    
    So, the obtained criterion that the segment $[x,y]$ is horizontal is equivalent to the analogous criterion that the segment $[f(x),f(y)]$ is horizontal.
\end{proof}

\begin{figure}[ht]
\center{\includegraphics[scale=0.5, width=300pt]{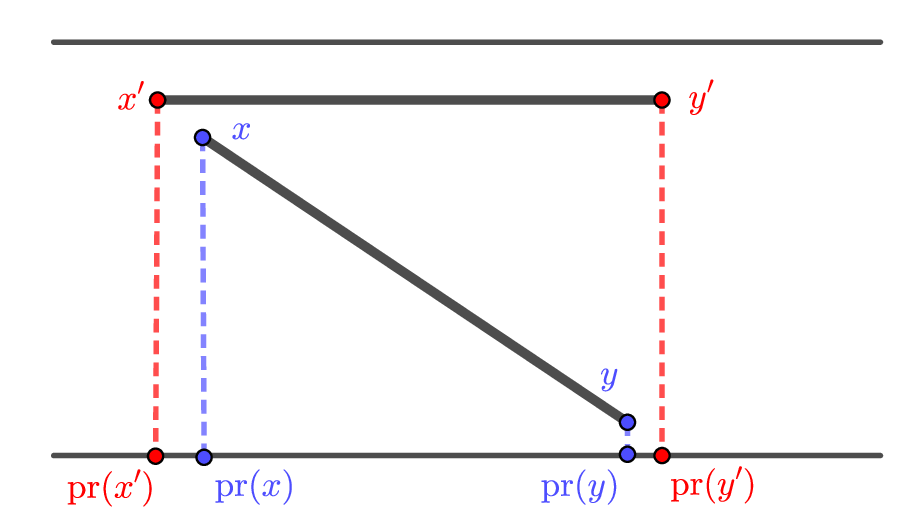}}
\caption{The segment $[x,y]$ of length $1$ is horizontal if and only if there are no such points $x',y'$ that $\rho_{G_1}(x',y')=1$ and the segment $[\pr(x'),\pr(y')]$ contains the points $\pr(x),\pr(y)$ strictly in the interior}
\label{fig:horizontal}
\end{figure}

Now we are ready to introduce a construction, that differs the unit distance graphs of $L_1$ and $L_2$ in the case $\eps_1, \eps_2 < 1$.

Consider $N, M \in \N$, such that $M > N$. Consider different points $a_0, a_1,...,a_N$, different points $b_0, b_1,... b_M$, and different points $c_1, c_2,...,c_M $ in $\R\times[0,\eps]$, satisfying the following properties.

(i) $a_0, ...,a_N$ lie on the same horizontal line $l$;

(ii) $||a_0-a_1||_2=||a_1-a_2||_2=....=||a_{N-1} - a_N||_2=1$;

(iii) $a_0=b_0, a_N=b_M$;

(iv) $||b_0-c_1||_2=||c_1-b_1||_2=||b_1-c_2||_2=...=||c_N - b_N||_2=1$.

(v) For $i = 1,2,...,M-1$ the projections $\pr(b_i), \pr(c_i)$, and the projection $\pr(c_M)$ lie strictly in the interior of the segment $[a_0, a_N]$;

(vi) For $i = 1,2,...,M-1$, the projection $\pr(b_i)$ is strictly between $\pr(b_{i-1})$  and $\pr(b_{i+1})$. For $i = 1,2,...,1$, the projection $\pr(c_i)$ is strictly between $\pr(b_{i-1})$ and $\pr(b_{i})$.

We say that such points $a_0, ..., a_N, b_0,...,b_M, c_0,...c_M$ form a \textit{$(N,M)$-comb}, see Fig. \ref{fig:comb} (we do not require that the points $b_0,b_1,...b_M$ lie on the same line and that the points $c_0,c_1,...,c_M$ lie on the same line, but this case is extreme).

\begin{figure}[ht]
\center{\includegraphics[scale=0.5, width=400pt]{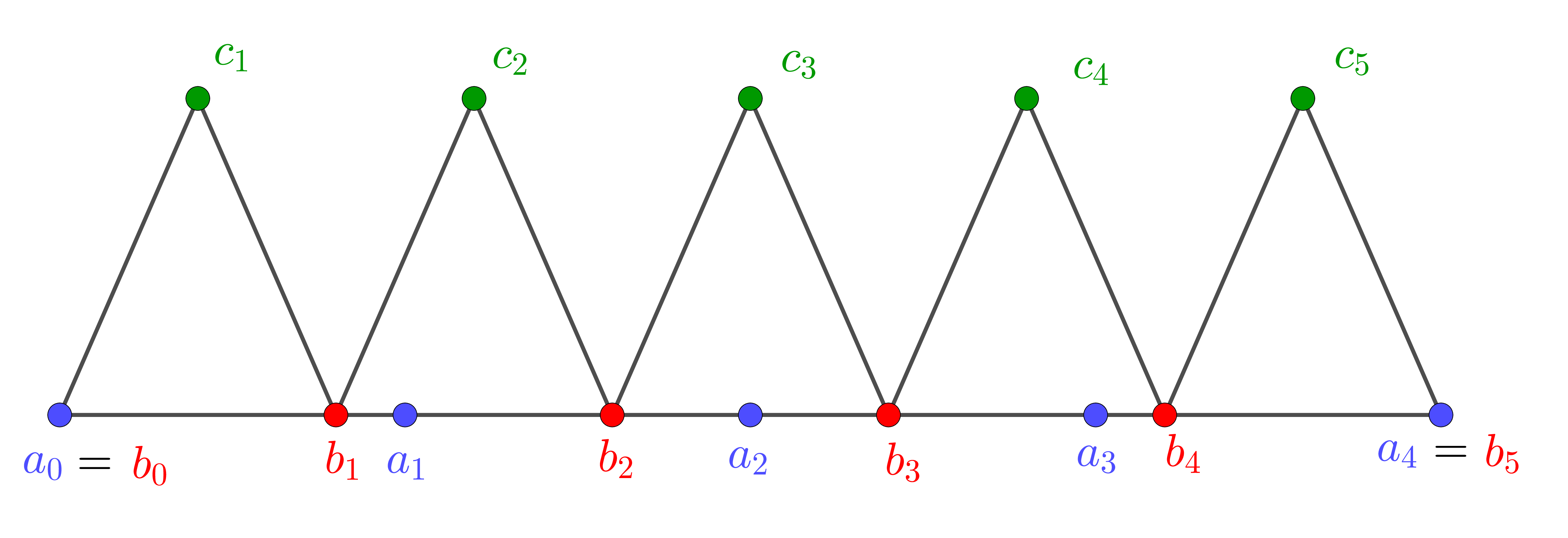}}
\caption{The extreme case of the $(4,5)$-comb}
\label{fig:comb}
\end{figure}

\begin{proposition}\label{prop:comb}
    Let $0 < \eps < 1$. Then the layer $\R \times [0,\eps]$ contains a $(N,M)$-comb if and only if $\eps \geq \sqrt{1-\frac{N^2}{4M^2}}$.
\end{proposition}

\begin{proof}
Let $a_0, ..., a_N, b_0,...,b_M, c_0,...c_M \in \R\times[0,\eps]$ form a comb.
The segment $[a_0,a_N]$ of this comb has length $N$. The points $\pr(b_1),\pr(b_2)$,$...$, $\pr(b_{M-1})$ divide this segment into $M$ parts. Then for at least one number $i\in [M-1]$ we have $||\pr(b_i)- \pr(b_{i+1})||_2 \leq \frac{N}{M}$. So, the triangle $b_i b_{i+1} c_{i+1}$ has two sides of length $1$ and is contained in a rectangle with length of the vertical sides $\eps$, and length of the horizontal sides at most $\frac{N}{M}$. We will find the minimum value of $\eps$ at which it is possible.

Let this rectangle be $ABCD$, and let $||A-B||_2=\eps, ||B-C||_2 \leq \frac{N}{M}$, $b_{i-1} \in AB, b_{i+1} \in CD$. By a parallel transport, we can assume without loss of generality that $b_{i-1} = A$ and $c_i$ is in the polygon $ABCb_i$. Denote by $E$ and $F$ the projections of $c_i$ onto $CD$ and $AD$, respectively; see Fig. \ref{fig:rectangle}. Then $$||D-c_i||_2 \geq \text{diam}\,EDFc_i \geq ||b_i-c_i||_2 = 1,$$ 
because $EDFc_i$ is a rectangle. Therefore,
$$||A-c_i||_2= 1, ||D-c_i||_2 \geq 1;$$
$$\min(||A-F||_2, ||F-D||_2) \leq \frac{N}{2M};$$
$$\eps = ||A-B||_2 \geq ||c_i-F||_2 \geq \sqrt{1-\frac{N^2}{4M^2}}.$$
On the other hand, obviously, the estimate $\eps = \sqrt{1-\frac{N^2}{4M^2}}$ is sufficient for the "extreme" case of comb, when $b_0,b_1,...b_M$ lies on the line $l$ and divide it into equal parts, see Fig. \ref{fig:comb}. Hence, this estimate is tight.
\end{proof}

\begin{figure}[ht]
\center{\includegraphics[scale=0.5, width=170pt]{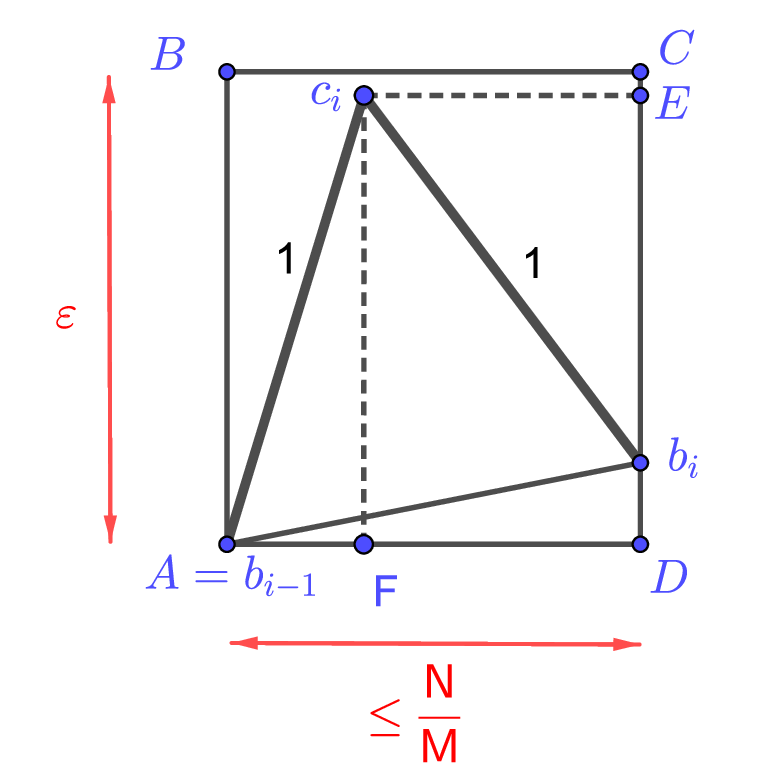}}
\caption{The rectangle $ABCD$ containing the triangle $b_{i-1}b_ic_i$.}
\label{fig:rectangle}
\end{figure}

\begin{lemma}\label{l:eps_less_one}
    Let $\eps_1, \eps_2 < 1$. Then $L_1$ contains a $(N,M)$-comb if and only if $L_2$ contains a $(N,M)$-comb.
\end{lemma}

\begin{proof}
    Let $a_0, ..., a_N, b_0,...,b_M, c_0,...c_M \in L_1$ form a comb. The properties (ii)-(iv) hold for $f(a_0),...,f(a_N)$,$f(b_0),...,f(b_N)$, $f(c_1)$,$...$,$f(c_M)$, because $f$ is an isomorphism of unit distance graphs. The property (i) holds due to Lemma \ref{l:horizontal}. The properties (v)-(vi) hold due to Lemma \ref{l:vert}. Hence, $f(a_0),...,f(a_N)$,$f(b_0),...,f(b_N)$,$f(c_1)$,$...$, $f(c_M)\in L_2$ form $(N,M)$-comb.
\end{proof}

In the case where $\eps_1 \geq 1$ or $\eps_2 \geq 1$, we also use an additional construction to reduce the problem to the comparison of the fractional parts of $\eps_1$ and $\eps_2$. For $m \in \N$, consider the set $\{(x,y) : x \in \Z, y \in \{0,1,\dots,m\}\}$ in $\R^2$. Let us call $m$-sandwich any set obtained from the given set by a shift by a real vector, as is shown on Fig. \ref{fig:sandwich}. We say that $x$ is a \textit{border point} of a $m$-sandwich $S$, if $x \in S$, and $x$ does not lie between any other two points of $S$ that form a vertical segment.

\begin{figure}[ht]
\center{\includegraphics[scale=0.5, width=250pt]{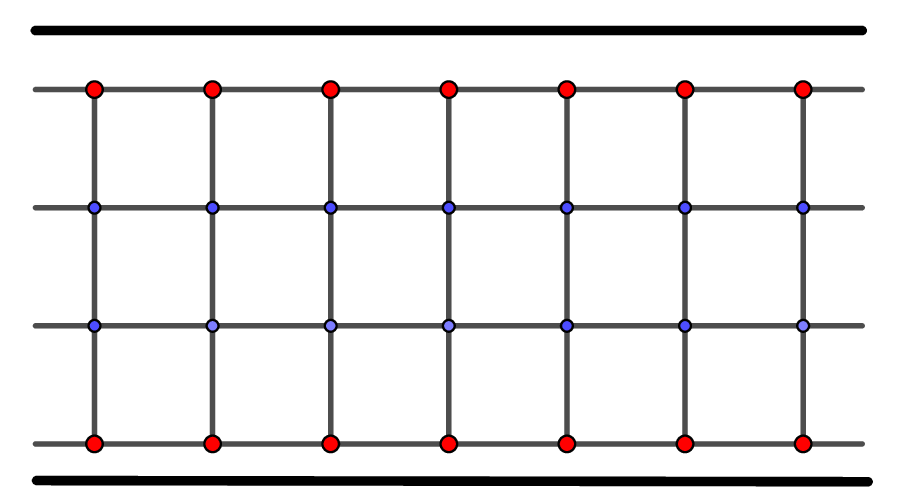}}
\caption{An example of $3$-sandwich. The border points are colored red.}
\label{fig:sandwich}
\end{figure}

By a corollary of Lemmas \ref{l:vert}, \ref{l:vert_xyz}, \ref{l:horizontal}, and Corollary \ref{cor:boundary}, we immediately get the following observation.

\begin{observation}\label{obs:sandwich}
Let $S \subset L_1$ be a $m$-sandwich. Then $f(S) \subset L_2$ also is a $m$-sandwich, and the mutual arrangement of the points of $L_2$ in the horizontal rows and vertical columns is the same in the following sense. If $x,y,z \in S$ are different points, $x \in [y,z]$, and the segment $[y,z]$ is vertical (horizontal), then $f(x) \in [f(y), f(z)]$, and the segment $[f(y),f(z)]$ is also vertical (horizontal). In particular, if $x$ is a border point of $S$, then $f(x)$ is a border point of $f(S)$.
\end{observation}

Denote by $m_i$ and $\delta_i$ the integer part and the fractional part of $\eps_i$, respectively:
$$m_i = \lfloor\eps_i \rfloor,\; \delta_i = \eps_i - m_i.$$
Consider the sets
$$\Delta_i = \{x = (x_1,x_2) \in L_i:\; x_2 \in [0, \delta_i)\cup(m_i, \eps_i]\}.$$

\begin{lemma}\label{l:delta}
    If $x \in \Delta_1$, then $f(x) \in \Delta_2$.
\end{lemma}

\begin{proof}
    Note that $x \in \Delta_i$ if and only if $x$ is a border point of a $m_1$-sandwich in $L_1$,  such that all the points of this sandwich are not contained in $\R\times\{0,\eps_i\}$. So, it is enough to apply Corollary \ref{cor:boundary} and Observation \ref{obs:sandwich}.
\end{proof}

We say that points $a_0,a_1,...,a_N, b_0,b_1,...,b_M, c_1,...,c_M \in L_i$ form a \textit{modified $(N,M)$-comb} if they are contained in $\Delta_i$, they satisfy the statements (i)-(vi) of the definition of $(N,M)$-comb, and $c_1,...,c_M \in \R \times \{0,\eps_i\}$.

\begin{lemma}\label{l:planar_end}
    Let $a_0,a_1,...,a_N, b_0,b_1,...,b_M, c_1,...,c_M \in \Delta_i$ form a modified $(N,M)$-comb. Then

    $(\alpha)$ all these points lie in $\R\times[0,\delta_1]$ or all these points lie in $\R\times[m_1,\eps_1]$;

    $(\beta)$ the points $f(a_0),...,f(a_N),f(b_0),...,f(b_N),f(c_1),...,f(c_M) \in L_2$ satisfy the properties (i)-(vi), and $f(c_1),...,f(c_M) \in \R \times \{0,\eps_2\}$;

     $(\gamma)$ all the points $f(a_0),...,f(a_N),f(b_0),...,f(b_N),f(c_1),...,f(c_M)$ lie in $\R\times[0,\delta_2]$ or all these points lie in $\R\times[m_2,\eps_2]$.
\end{lemma}

\begin{proof}
    $(\alpha)$. 
    For definiteness, let $c_1 \in \R\times\{0\}$.
    
    For any point $t \in  \R\times[m_1,\eps_1]$, we have $||t-c_1||_2 > 1$. Since $||b_0 - c_1||_2 = ||b_1-c_1||_2=1$ and $b_0, b_1 \in \Delta_1$, we get $b_0,b_1 \in \R\times[0,\delta_1]$.

    For any point $t \in  \R\times\{\eps_1\}$, we have $||t-b_1||_2 > 1$. Since $||b_1 - c_2||_2 = 1$, and $c_2 \in \R\times\{0,\eps_1\}$, we get $c_2 \in \R\times\{0\}$.

    Thus, by induction on $j$ we can analogously show that all the points $b_j,c_j$ are in $\R\times\{0\}$.

    Then all the points $a_0,a_1,...,a_N$ lie in $\R\times\{0\}$, because they are on the same horizontal line.

    $(\beta)$. The properties (ii),(iii),(iv) hold for $f(a_0),...,f(a_N)$,$f(b_0),...,f(b_N)$, $f(c_1)$,$...$,$f(c_M)$, because $f$ is an isomorphism of unit distance graphs. The property (i) holds due to Lemma \ref{l:horizontal}. The properties (v)-(vi) hold due to Lemma \ref{l:vert}. The property $f(c_1),...,f(c_M) \in \R \times \{0,\eps_2\}$ holds due to Corollary \ref{cor:boundary}.

    $(\gamma)$ Due to $(\beta)$, the proof is the same as the proof of $(\alpha)$.
\end{proof}

Since the widths of the strip $\R\times[0,\delta_1]$ and $\R\times[m_1,\eps_1]$ are equal to $\delta_1$, Lemma \ref{l:planar_end} and Proposition \ref{prop:comb} immediately give the following corollaries.

\begin{corollary}\label{cor:planar_end_1}
    Let $m_i \geq 1, 0 < \delta_i$. Then $L_i$ contains a modified $(N,M)$-comb if and only if $\delta_i \geq \sqrt{1-\frac{N^2}{4M^2}}$.
\end{corollary}

\begin{corollary}\label{cor:planar_end_2}
    Let $m_1 = m_2 \geq 1$. Then $L_1$ contains a modified $(N,M)$-comb if and only if $L_2$ contains a modified $(N,M)$-comb.
\end{corollary}

Now we are ready to complete the proof of Theorem \ref{t:planar}. Without loss of generality, we assume that $\delta_1 \geq \delta_2$.

Case I. Let $m_1 > m_2$. That contradicts Observation \ref{obs:sandwich}, because $L_1$ contains $m$-sandwich, and $L_2$ does not.

Case II. Let $m_1 = m_2 = 0, \eps_1 > \eps_2$. Then there exist integer numbers $N,M$ such that $\eps_1 > \sqrt{1-\frac{N^2}{4M^2}} > \eps_2$. Due to Proposition \ref{prop:comb}, $L_1$ contains a $(N,M)$-comb, but $L_2$ does not. That contradicts Lemma \ref{l:eps_less_one}.

Case III. Let $m_1 = m_2 > 0, \delta_1 > \delta_2$. Then there exist integer numbers $N,M$ such that $\delta_1 > \sqrt{1-\frac{N^2}{4M^2}} > \delta_2$. Due to Corollary \ref{cor:planar_end_1}, $L_1$ contains a modified $(N,M)$-comb, but $L_2$ does not. That contradicts Corollary \ref{cor:planar_end_2}.

Therefore, we get $\eps_1 = \eps_2$, which completes the proof of Theorem \ref{t:planar}.

\section{Proof of Theorem \ref{t:main}}
Let $n,m \in \N, n \geq 2; p \in (0,+\infty); \eps_1, \eps_2 \in (0,+\infty)$. Consider the layers $L_1 = L(n, m, p, \eps_1)$, $L_2 = L(n, m, p, \eps_1)$, and denote the unit distance graphs of these layers by $G_1$ and $G_2$, respectively. Denote by $\rho_i(x,y)$ the length of shortest path between two vertices $x,y$ in the graph $G_i$. 

Suppose $f:L_1\ra L_2$ is an isomorphism of the graphs $G_1, G_2$. We will prove that this implies $\eps_1=\eps_2$.

For definiteness, we take $n$-dimensional real space $\H$, $m$-dimensional real space $\V$, and $(n+m)$-dimensional real space $\T$ such that, for $i\in \{1,2\}$,
$$L_i = \H \times [0,\eps]^m \subset \H \times \V = \T.$$

By analogy to Theorem \ref{t:planar}, we say that any subspace $P$ in $\T$ parallel to the subspace $\H$ is \textit{horizontal}, and any subspace $P$ in $\T$ parallel to the subspace $\V$ is \textit{vertical}.

\begin{proposition}\label{prop:dist2}
    Given $i\in \{1,2\}, x,y \in L_i$, we have $||x-y||_p = 2$ if and only if there exists exactly one point $z\in L_i$ such that $||x-z||_p = ||y-z||_p = 1$.

    In this case, $z$ is the middle of the segment $[x,y]$.
\end{proposition}

\begin{proof}
Note that $L_i$ is a convex set.
    
    If $||x-y||_p = 2$, then obviously the unique desired point is $z = \frac{x+y}{2} \in L_i$. If $||x-y||_p > 2$, then there are no such points.

    We are only left to check that if $||x-y||_p <2$ then there exist at least two points $z',z''$ such that $||x-z'||_p = ||y-z'||_p = ||x-z''||_p = ||y-z''||_p = 1$. Indeed, let us consider a horizontal $2$-dimensional plane $P$ passing through the point $\frac{x+y}{2}$ (it exists because $n\geq 2$). Then there exists a line $l \subset P$ orthogonal to the line $xy$ such that $\frac{x+y}{2} \in l$. Then $||x-\frac{x+y}{2}||_{p} = ||y-\frac{x+y}{2}||_{p} < 1$, and $\frac{x+y}{2}$ divides $l$ into two rays. So, we can symmetrically take $z',z''$  with the necessary property on these rays.
\end{proof}

\begin{corollary}\label{cor:distk}
    For any $k \in \N$ and $x,y\in L_1$, we have $||x-y||_p=k$ if and only if $||f(x)-f(y)||_p=k$.
\end{corollary}

\begin{proof}
For $k = 0,1$ the statement obviously holds because $f$ is an isomorphism of the graphs $G_1,G_2$. For $k=2$ it holds due to Proposition \ref{prop:dist2}.

    Let $k>2$. If $||x-y||_p=k$, then there exists a sequence $x_0= x, x_1,x_2,...,x_k = y$ of points on the same line in $L_1$ that satisfy $||x_i-x_{i+1}||_p=1$ for $i=0,1,...,k-1$. For any $i \in \{0,1,...,k-2\}$, we have $||x_i-x_{i+2}||_p = 2$. Then $x_{i+1}$ is the unique point $z$ such that $||x_i-z||_p=||x_{i+2}-z||_p=1$. Since $f$ is an isomorphism of the graphs $G_1,G_2$, $f(x_{i+1})$ is the unique point $z$ such that $||f(x_i)-z||_p=||f(x_{i+2})-z||_p=1$. Hence, by Proposition \ref{prop:dist2}, $f(x_{i+1})$ is the middle of the segment $[f(x_i) f(x_{i+2})]$, and $||f(x_i)-f(x_{i+1})||_p = ||f(x_{i+1})-f(x_{i+2})||_p = 1$ for $i=0,1,...,k-2$. Therefore, $||f(x)-f(y)||_p = ||f(x_0)-f(x_k)||_p = k$.
\end{proof}

\begin{corollary}\label{cor:horizontal2}
Let $a,b \in L_1$ be two different points such that the line $ab$ is horizontal. Then the line $f(a)f(b)$ is horizontal.
\end{corollary}

\begin{proof}
 For any $i \in \{1,2\}$, $\textbf h \in \H$, $\textbf v \in \V \setminus \{0\}$, there exists a sufficiently large $c \in \R$ such that $c\textbf v \notin [0,\eps]^m$, and consequently, $\textbf h+\textbf v \notin L_i$. Hence, if $x_1,x_2,...,x_k,...\in L_i$ is a sequence of different points with the property
$$(*) \forall j \in \N \; \exists! x = x_{j+1} \text{\;such that\;} ||x_j - x||_p = ||x_{j+2}-x||_p = 1,$$
then, due to Proposition \ref{prop:dist2}, all the points $x_j$ lie on the same horizontal line.

\begin{figure}[ht]
\center{\includegraphics[scale=0.5, width=250pt]{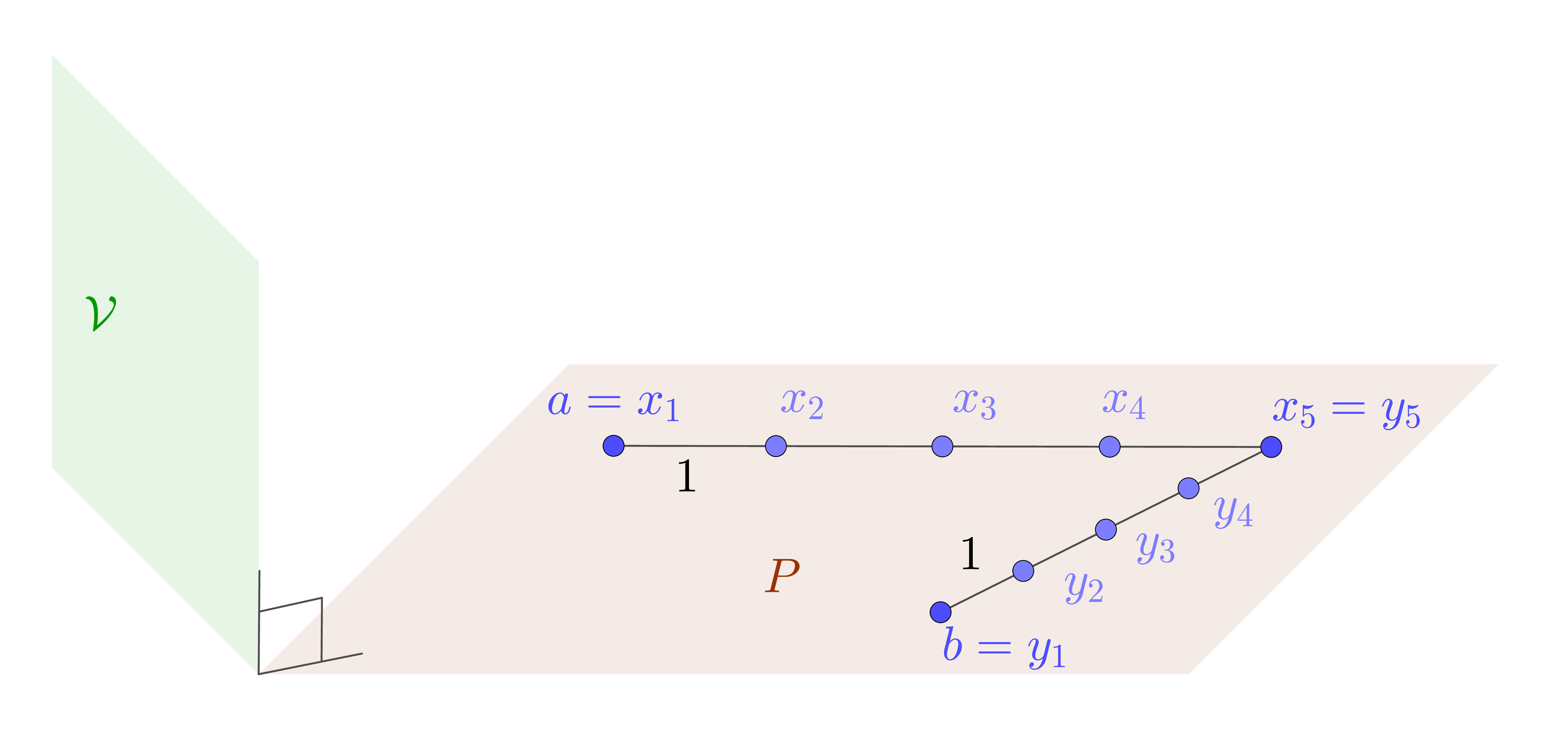}}
\caption{The condition of horizontality of a segment $ab$ in $L_1$.}
\label{fig:dist2}
\end{figure}

Since $n \geq 2$, there exists a horizontal $2$-dimensional plane $P\subset L_1$ containing $a,b$. Thus,
there exists $k \in \N$ and two sequences of different points $\{x_1=a,x_2,...,x_k, x_{k+1},..\}$, $\{y_1=b, y_2,..,y_k=x_k, y_{k+1},...\}$ on $P$, such that $(*)$ holds for the sequence of $x_j$ and holds for the sequence of $y_j$ (see Fig. \ref{fig:dist2}). Since $f$ is an isomorphism of graphs, $(*)$ also holds for the sequence of $f(x_j)$ and for the sequence of $f(y_j)$.
Hence, the lines $f(x_1)f(x_2)$ and $f(y_1)f(y_2)$ are horizontal, and, consequently, $f(a)f(b)$ is also horizontal.
\end{proof}

For $x\in L_j$ denote by $\pr(x)$ the orthogonal projection of $x$ onto the subspace $\H$. Also denote by $H_x$ ($V_x$) the unique horizontal (vertical) $n$-dimensional ($m$-dimensional) subspace in $L_2$ passing through $x$. Note that for any $x \in L_i$ we have $H_x \subset L_i$, but $V_x \not\subset L_i$.
Denote by $(\cdot, \cdot)$ the standard Euclidean scalar product in the spaces $\H, \V, \T$.

The following lemma is a multidimensional analogue of Lemma \ref{l:main1}.

\begin{lemma}\label{l:main2}
    Let $x_1,...,x_{n+1}, y \in L_i$ be different points. Then the implications $(\widehat\Omega) \Leftrightarrow (\widehat\Gamma) \Leftrightarrow (\widehat \Gamma')$
    %and $\neg(\widehat\Omega) \Leftrightarrow (\widehat\Gamma^-)$
    hold for the following conditions.
    
    $(\widehat\Gamma)$ For any horizontal ray $l \subset L_i$  from the point $y$, there exist $r > 0$ and $k \in [n+1]$, such that for any $z \in l$ with $||z-y||_p>r$ we have $||z-x_k||_{p} < ||z-y||_{p}$.

    $(\widehat\Gamma')$ There exists a constant $r > 0$ such that for any $z \in \H_y$ with $||z-y||_p>r$ we have $\min\limits_{k\in[n+1]}||z-x_k||_{p} < ||z-y||_{p}$.

    %$(\widehat\Gamma^-)$ There exists a horizontal ray $l \subset L_i$ from the point $y$, and a sequence of different points $z_0 = y, z_1,z_2,... \in l$, such that for any $i \in \N$ we have $||z_i-z_{i+1}||_{p}=1$, and $\min\limits_{k\in [n+1]}||z_i-x_k||_{p} > ||z_i-y||_{p}$.

    $(\widehat \Omega)$. The interior of the convex hull of the points $\pr(x_1),...\pr(x_{p_j+1})$ in the space $\H$ is not empty and contains $\pr(y)$.
\end{lemma}
\begin{proof}[Proof of Lemma \ref{l:main2}.]

\textit{Part 1. The implications $(\widehat\Omega) \Leftrightarrow (\widehat\Gamma)$}.

 As is shown in the proof of Corollary \ref{cor:horizontal2}, any straight line $l \subset L_j$ is horizontal, and, consequently, is contained in $\H_y$. Consider the $(n-1)$-dimensional surface $S_{1,\H} = \{\textbf u \in \H: ||\textbf u||_{p} = 1 \}$ in the space $\H$. Any horizontal ray from $y$ can be represented as $\{y+R \textbf u: R\geq 0\}$ for some $\textbf u \in \H$.
  Thus,
$$\big(I\big).\; (\widehat\Gamma) \Leftrightarrow \forall \textbf u \in S_{1,\H}\;  \exists k \in [n+1]\; \exists r>0 \; \forall R>r:$$
$$||x_k-(y+R\textbf u)||_{p}<||y-(y+R\textbf u)||_p.$$

Consider the $(n+m-1)$-dimensional closed set $B_1 = \{\textbf u \in \T: ||u||_{p} \leq 1\}$ (the unit ball generated by the $l_p$-norm) in the space $\T$. Since $l_p$ is a smooth convex linear norm ($p \in (1,+\infty)$), the ball $B_{1}$ is a convex set with a smooth boundary $\partial B_{1}$.
Hence, for any $\textbf u \in S_{1,\H}$, there exists an unique $(n+m-1)$-dimensional supporting hyperplane $P_{\textbf u}$ to $\partial B_{1}$ at the point $\textbf u$.

The ball $B_1$ of the $l_p$-norm is a symmetric set with respect to the hyperplane $\H$. Then, since $n \geq 2$, the set $\pr(B_1)$ is a convex set in $\H$ with a smooth boundary $\partial (\pr(B_1))$, and $\pr(P_{\textbf u}) \subset \H$ is a $(n-1)$-dimensional supporting hyperplane to the surface $\partial (\pr(B_1)) = S_{1,\H}$ in the space $\H$.

Hence, there exists $s_{\textbf u} \in \H$ such that
$$P_{\textbf u} = \{p \in \T: (s_{\textbf u},p-s_{\textbf u})_E = 0\}.$$

The hyperplane $P_{\textbf u}$ divides the space $\T$ into two half-spaces $\T_{\textbf u}^+$, $\T_{\textbf u}^-$ such that $0 \in \T_{\textbf u}^-$, and
$$\T_{\textbf u}^+ =  \{\textbf w \in \T: (s_{\textbf u}, \textbf w-s_{\textbf u})_E \geq 0\};$$
$$\T_{{\textbf u}}^- =  \{\textbf w \in \T: (s_{\textbf u}, \textbf w-s_{\textbf u})_E < 0\}.$$

The subspace $\pr(P_{\textbf u})$ divides the space $\H$ into two parts $\pr(\T_{{\textbf u}}^+), \pr(\T_{\textbf u}^-)$. For any ${\textbf u}\in S_{1,\H}$ and $\textbf w \in \T$ represented as $\textbf w = \textbf h+\textbf v, \textbf h \in \H, \textbf v \in \V$, we have 
 $$(s_{\textbf u},\textbf w-s_{\textbf u})_E = (s_{\textbf u},\textbf h+\textbf v-s_{\textbf u})_E = (s_{\textbf u},\textbf h-s_{\textbf u})_E + (s_{\textbf u},\textbf v)_E = (s_{\textbf u},\textbf h-s_{\textbf u})_E,$$ because $s_{\textbf u} \in \H$ and $\textbf v \in V$ are orthogonal vectors. Hence, 
$$\big(II\big).\;\textbf w \in \H_{\textbf u}^+ \Leftrightarrow \pr(\textbf w) \in \pr(H_{\textbf u}^+).$$

For any $\textbf u \in S_{1,\H}, a \in \T$, we have: 
$$\big(III\big).\; \forall r>0 \; \exists R>r\;\; \textbf u+\frac{a}{R} \notin B_1 \Leftrightarrow a \in \T_{\textbf u}^+.$$
For $k \in [n+1]$, denote the vector $y-x_k$ by $\textbf a_k$. Any point $z \in \H_y \setminus \{y\}$ can be represented as $z = y+r_z \cdot \textbf u_z$, where $\textbf u_z \in S_{1,\H}, r_z > 0$. Then
$$\big(IV\big).\; ||x_k-z||_{p}>||y-z||_{p} \Leftrightarrow||y-(z+\textbf a_k)||_{p} > ||y-z||_{p} \Leftrightarrow$$
$$\Leftrightarrow r_z \textbf u_z+\textbf a_k \notin r_z \cdot B_1 \Leftrightarrow \textbf u_z+\frac{\textbf a_k}{r_z} \notin B_1.$$

Condition $(\widehat \Omega)$ is equivalent to the following (see Fig. \ref{fig:lemma_main2_omega} for an illustration):
$$\big(V\big). \;(\widehat\Omega) \Leftrightarrow \forall \textbf u \in S_{1,\H}\; \exists k \in [n+1]:\; (\pr(\textbf a_k), \textbf u) < 0.$$

\begin{figure}[ht]
\center{\includegraphics[scale=0.5, width=250pt]{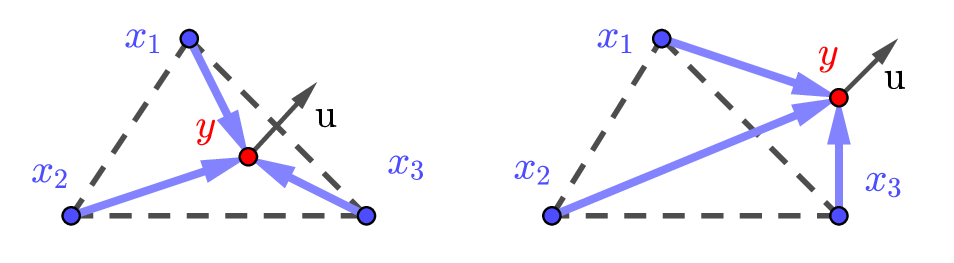}}
\caption{a) The case $n = 2$, $(\Omega)$ holds. b) The case $n = 2$, $(\Omega)$ does not hold.}
\label{fig:lemma_main2_omega}
\end{figure}

Thus,
$$\neg(\widehat\Gamma) \Leftrightarrow \text{ by $(I)$ } \Leftrightarrow$$
$$\Leftrightarrow \exists \textbf u \in S_{1,\H}\;  \forall k \in [n+1]\; \forall r>0 \; \exists R>r:$$
$$||x_k-(y+R\textbf u)||_{p}<||y-(y+R\textbf u)||_p\Leftrightarrow$$
$$\Leftrightarrow \text{ by $(IV)$ } \Leftrightarrow  \exists \textbf u \in S_{1,\H}\;  \forall k \in [n+1]\; \forall r>0 \; \exists R>r:\; \textbf u+\frac{\textbf a_k}{R} \notin B_1\Leftrightarrow $$
$$\Leftrightarrow \text{ by $(III)$ } \Leftrightarrow \exists \textbf u \in S_{1,\H}\;\forall k \in [n+1]:\; \textbf a_k \in \T_{\textbf u}^+\Leftrightarrow $$
$$\Leftrightarrow \text{ by $(II)$ } \Leftrightarrow \exists \textbf u \in S_{1,\H}\;\forall k \in [n+1]:\; \pr(\textbf a_k) \in \pr(\T_{\textbf u}^+)\Leftrightarrow$$
$$\Leftrightarrow \exists \textbf u \in S_{1,\H}\;\forall k \in [n+1]:\; (\pr(\textbf a_k),\textbf u)_E \geq 0 \Leftrightarrow \text{ by $(V)$ } \Leftrightarrow\neg(\widehat\Omega).$$

\begin{figure}[ht]
\center{\includegraphics[scale=0.5, width=320pt]{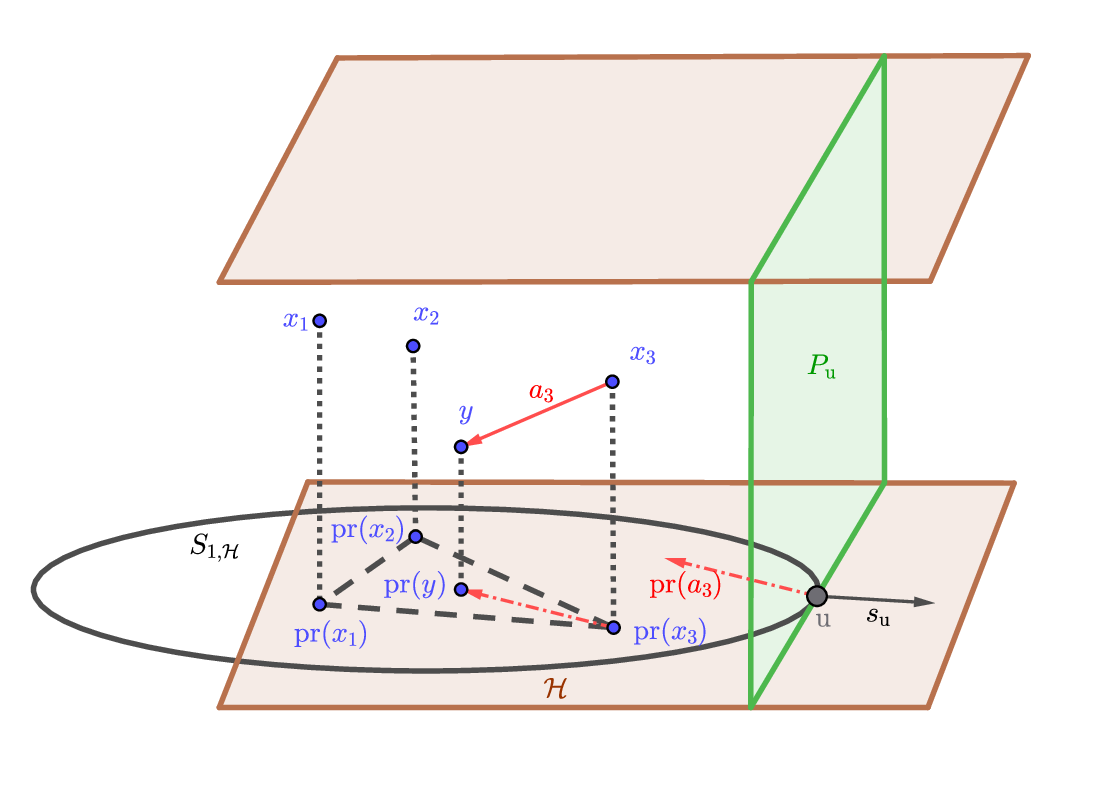}}
\caption{The case $n = 2$, $(\widehat\Omega), (\widehat\Gamma)$ hold.}
\label{fig:main_lemma_2_example_1}
\end{figure}

\begin{figure}[ht]
\center{\includegraphics[scale=0.5, width=290pt]{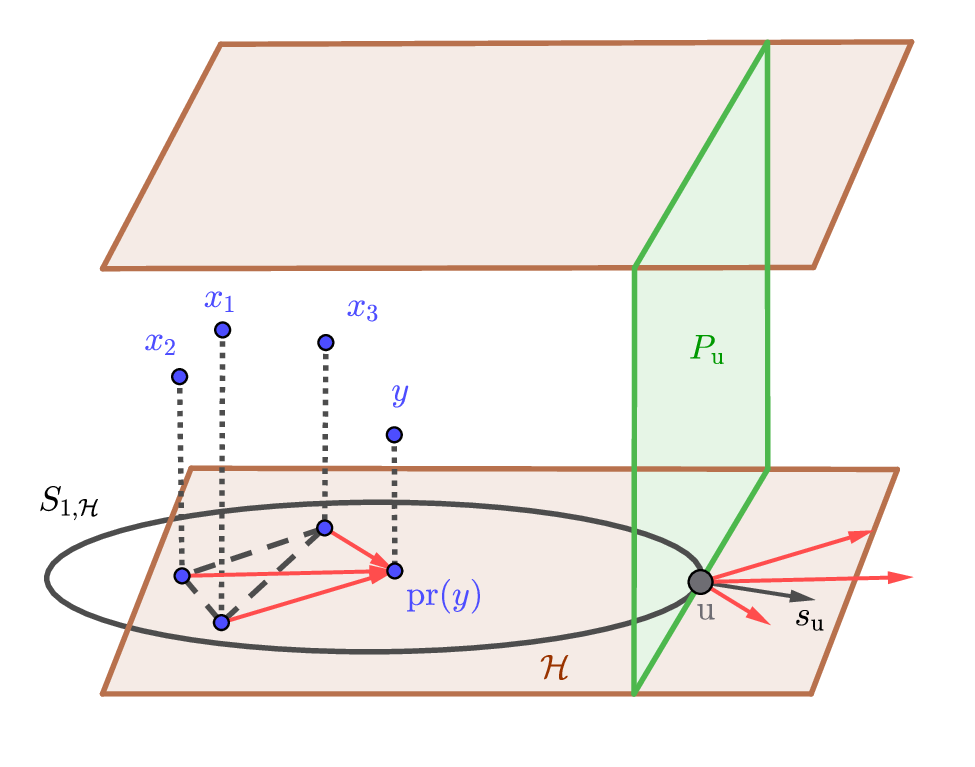}}
\caption{The case $n = 2$, $\neg(\widehat\Omega), (\widehat\Gamma^-)$ hold.}
\label{fig:main_lemma_2_example_2}
\end{figure}

Therefore, $(\widehat\Omega) \Leftrightarrow (\widehat\Gamma)$
(see Figs. \ref{fig:main_lemma_2_example_1}, \ref{fig:main_lemma_2_example_2} for an illustration of the key idea).
%Since conditions $(\widehat\Gamma)$ and $(\widehat\Gamma^-)$ are mutually exclusive, we are only left to verify the implication $\neg(\widehat\Omega) \Rightarrow (\widehat\Gamma^-)$.

\textit{Part 2. The implications $(\widehat\Gamma) \Leftrightarrow (\widehat\Gamma')$}.

The implication $(\widehat\Gamma') \rightarrow (\widehat\Gamma)$ is obvious.

Let $(\widehat\Gamma)$ holds. Then, for any fixed $\textbf{u} \in S_{1,\H}$ and horizontal ray $l_{\textbf u} = \{y+R\textbf{u} \in \H_y :\; R \geq 0\}$, there exists a non-empty set $\mathcal R_\textbf{u}$ of constants $r > 0$ such that for any $z \in l_{\textbf u}$ with $||z-y||_p > r$ we have
$$\min\limits_{k\in[n+1]} ||z-x_k||_p < ||z-y||_p.$$

Denote $r(\textbf{u}) = \inf\limits_{r \in \mathcal{R}_\textbf{u}} r$. By the construction of $\mathcal{R}_\textbf{u}$, we have $r(\textbf{u}) \in \mathcal{R}_\textbf{u}$.

Obviously, $z \in \H_y$ if and only if there exists ${\textbf u} \in S_{1,\H}$ such that $z \in l_{\textbf u}$.

Since the functions $||z-y||_p$ and $||z-x_k||_p$ for $k \in [n+1]$ are continuous, the function $r(\textbf{u})$ is also continuous on the set $S_{1,\H}$. Then, due to compactness of $S_{1,\H}$, there exists $r = \max\limits_{\textbf{u} \in S_{1,\H}} r(\textbf{u}) < +\infty$. Then $(\widehat\Gamma')$ holds.
%This completes the proof (see Figs. \ref{fig:main_lemma_2_example_1}, \ref{fig:main_lemma_2_example_2} for an illustration of the key idea).

\end{proof}

Now we will prove an analogue of Proposition \ref{prop:eq}.

\begin{proposition}\label{prop:C}
     Given $i\in \{1,2\}$, there exists a constant $C_i=C_i(\eps_i)>0$, such that for any $a,b \in L_i$, if $\rho_{i}(a,b) > C_i$ then $\rho_{i}(a,b)=\left\lceil ||a-b||_{p}\right\rceil$.
\end{proposition}

\begin{proof}
If $||a-b||\in \N$, then the statement is obvious.

Let $||a-b||_p = m+\delta, m\in \N, m>10, \delta \in (0, 1)$.
 
Consider the set $S = \{x \in L_i: ||x-a||_p=1\}$.

Let $x_1 = [a,b] \cap S$. We have $||x_1-b||_p = m-1+\delta < m$.

Denote by $B = \{x \in \T: ||x||_p=1 \}$ the unit ball in the metric $l_p$ with center $0$ and radius $1$. Denote by $\partial B$ its boundary, it is a smooth surface. For any $u \in \partial B$ consider the maximal angle $\theta(u) \in (0,\pi)$ between the vector $-u$ and a ray $l$ from the point $u$ such that $l$ is contained in the supporting hyperplane to $\partial B$ at the point $u$ (the maximum exists due to compactness reasoning). Since $\partial B$ is compact, the number $\theta = \sup_{u \in \partial B} \theta(u)$ is well defined and is less than $\pi$.

Consider the ball
$$B_{b,a} = \{x \in \T: ||x-b||_p \leq ||a-b||_p\} = ||a-b||_p\cdot B+b.$$ 
By construction of the number $\theta$, for any $x \in \T$ we have
$$\angle bax > \theta \Rightarrow x \notin B_{b,a} \Rightarrow ||b-x||_p>||b-a||_p>m.$$

For any sufficiently large $m$, the angle between the line $ab$ and the subspace $\H_a$ is less than $\pi-\theta$. Thus, there exists a ray $l\subset \H_a$ from the point $a$ such that $\angle bax > \theta$. Take the point $x_2 = l \cap S$ (it exists because $\H_a \subset L_i$). We have $\angle bax_2 > \theta \Rightarrow ||b-x_2||_p > m$.

Therefore, we have two points $x_1,x_2 \in S = \{x \in L_i: ||x-a||_p=1\}$, such that
$$||x_1-b||_p < m < ||x_2-b||_p.$$

Since $n\geq 2$, the set $S$ is connected. The function $\lambda(x) = ||x-a||_p$ is continuous on the set $S$, and $\lambda(x_1) < m < \lambda(x_2)$. Hence, there exists a point $x_3 \in S \subset L_i$, such that $\lambda(x_3)=m$.

Then 
$$\rho_{i}(a,b) \geq \rho_{i}(b,x_3)+\rho_{i}(a,x_3)=||b-x_3||_p + ||a-x_3||_p = m+1;$$
$$\rho_{i}(a,b) \leq ||b-a||_p = m+\delta.$$

Hence, $\rho_{i}(a,b) = m+1  = \left\lceil ||a-b||_{p}\right\rceil$.
\end{proof}

\begin{corollary}\label{cor:C}
    There exists a constant $C=C(\eps_1,\eps_2)=\max(C_1,C_2)>0$, such that for any $i\in \{1,2\}, a,b \in L_i$, if $\rho_{i}(a,b) > C$ then $\rho_{i}(a,b)=\left\lceil ||a-b||_{p}\right\rceil$.
\end{corollary}

By Corollaries \ref{cor:horizontal2} and \ref{cor:C}, we immediately get the following corollary.

\begin{corollary}\label{cor:gamma}
    Let $x_1,...,x_{n+1},y \in L_1$ be different points. Then condition $(\widehat\Gamma')$ of Lemma \ref{l:main2} holds for $x_1,...,x_{n+1},y $ if and only if it holds for $f(x_1),...,f(x_{n+1}),f(y)$.
\end{corollary}

Thus, we can show that $f$ preserves the "vertical" sets $V_x$.

\begin{lemma}\label{l:vert2}
Given $x_1,x_2 \in L_1$, we have $x_2 \in V_{x_1}$ if and only if $f(x_2) \in V_{f(x_1)}$.
\end{lemma}

\begin{proof}
Since $x_2 \in V_{x_1}$, we get $\pr(x_1) = \pr(x_2)$. Then there are no such points $x_3,...,x_{n+1}$ that the interior of the convex hull of $x_1,...,x_{n+1}$ in $\H$ is not an empty set. This implies that there are no points $x_3,...,x_{n+1},y \in L_1$ such that condition $(\widehat\Omega)$ of Lemma \ref{l:main2} holds for $x_1,...,x_{n+1},y$. Then, by Lemma \ref{l:main2}, there are no such points $x_3,...,x_{n+1},y\in L_1$ that condition $(\widehat\Gamma')$ holds for $x_1,...,x_{n+1},y$. Hence, due to Corollary \ref{cor:gamma}, there are no such points $x_3,...,x_{n+1},y\in L_1$ that condition $(\widehat\Gamma')$ holds for $f(x_1),...,f(x_{n+1})$, $f(y)$. Hence, there are no such points $x'_3,...,x'_{n+1},y' \in L_2$ that $(\widehat\Omega)$ holds for $f(x_1)$, $f(x_2)$, $x_3'$, $...,x_{n+1}',y'$, and consequently $\pr(f(x_1)) = \pr(f(x_2))$. Then $f(x_2) \in V_{f(x_1)}$.
\end{proof}

Now we can also show that $f$ preserves the horizontal lines and the order of the points on the horizontal lines.

\begin{lemma}\label{l:hor_line}
    Let $x_1,x_2,x_3 \in L_1$ be different points on the same horizontal line, and let $y \in (x,z)$. Then $f(x_1),f(x_2),f(x_3) \in L_2$ are different points on the same horizontal line and $f(x_2) \in (f(x_1),f(x_3))$.
\end{lemma}

\begin{proof}
The projections $\pr(x_1),\pr(x_2), \pr(x_3)$ are pairwise different because $x_1,x_2,x_3$ are different points on the same $n$-dimensional horizontal hyperplane $H_{x_1}$. Corollary \ref{cor:horizontal2} shows that $f(x_2),f(x_3) \in H_{f(x_1)}$. Due to Lemma \ref{l:vert2}, the projections $\pr f(x_1),\pr f(x_2),\pr f(x_3)$ are different. 

Since $x_1,x_2,x_3$ are on the same line $l$, then $\pr(x_1),\pr(x_2), \pr(x_3)$ are on the same horizontal line. Then there are no such points $x_4,...,x_{n+1}\in L_1$ that the interior of the convex hull of $x_1,...,x_{n+1}$ in $\H$ is an empty set (in the case $n=2$, the interior of the convex hull of $x_1,...,x_3$ in $\H$ is an empty set). Then there are no points $x_4,...,x_{n+1},y\in L_1$ such that condition $(\widehat\Omega)$ of Lemma \ref{l:main2} holds for $x_1,...,x_{n+1},y$. Then, by Lemma \ref{l:main2}, there are no points $x_4,...,x_{n+1},y\in L_1$ such that condition $(\widehat\Gamma')$ of Lemma \ref{l:main2} holds for $x_1,...,x_{n+1},y$. Due to Corollary \ref{cor:gamma}, there are no points $x_4,...,x_{n+1},y\in L_1$ such that condition $(\widehat\Gamma')$ of Lemma \ref{l:main2} holds for $f(x_1),...,f(x_{n+1}),f(y)$. Hence, there are no such points $x'_4,...,x'_{n+1},y'\in L_2$ that condition $(\widehat\Omega)$ of Lemma \ref{l:main2} holds for $f(x_1),f(x_2),f(x_3),x'_4,...,x'_{n+1},y'$. Hence, $f(x_1),f(x_2),f(x_3)$ are on the same horizontal line.

Now we will prove that $f(x_2)\in [f(x_1),f(x_3)]$.

Take the constant $C$ as guaranteed by Corollary \ref{cor:C}.

Denote the set
$$W_{x_1,x_2,x_3} = \{w \in l: \min_{i \in \{1,2,3\}}\rho_1(w,x_i) > \max_{i<j;\, i,j \in \{1,2,3\}} \rho_1(x_i,x_j)+C+1 \}.$$

For any $w \in W_{x_1,x_2,x_3}$ and $i \in \{1,2,3\}$ we have $\rho_1(w,x_i) > C$. Then, by Corollary \ref{cor:C},
$$\forall w \in W_{x_1,x_2,x_3}:\; ||w-x_i||_p \geq \rho_1(w,x_i)-1 \geq$$
$$\geq \max_{i<j;\, i,j \in \{1,2,3\}} \rho_1(x_i,x_j)+C \geq \max_{i<j;\, i,j \in \{1,2,3\}} ||x_i-x_j||_p+C.$$

Hence, any point $w \in W_{x_1,x_2,x_3}$ is not contained in the segment $[x_1,x_3]$.

Thus, by obvious geometric reasoning,
$$\forall w \in W_{x_1,x_2,x_3}:\; ||w-x_2||_p \geq \min_{i\in\{1,3\}} ||w-x_i||_p;$$
$$\forall w \in W_{x_1,x_2,x_3}:\; ||w-x_2||_p > C.$$

Hence, due to Corollary \ref{cor:C},
$$(*)\forall w \in W_{x_1,x_2,x_3}:\; \rho_1(w,x_2) \geq \min_{i\in\{1,3\}} \rho_1(w,x_i).$$

The first part of this proof shows that the set $f(l)$ is a horizontal line in $L_2$. Analogously to $W_{x_1,x_2,x_3}$, denote the set
$$\; W'_{x_1,x_2,x_3} = \{w \in f(l): \min_{i \in \{1,2,3\}}\rho_2(w,f(x_i)) > $$
$$>\max_{i<j;\, i,j \in \{1,2,3\}} \rho_2(f(x_i),f(x_j))+C+1 \}.$$

Since $f$ is an isomorphism of the graphs $G_1,G_2$, for any $ a,b \in \T$ we have $\rho_1(a,b) = \rho_2(f(a),f(b))$. Then
$$f(W_{x_1,x_2,x_3}) = W'_{x_1,x_2,x_3};$$
$$(*')\; \forall w \in W'_{x_1,x_2,x_3}:\; \rho_2(w,f(x_2)) \geq \min_{i\in\{1,3\}} \rho_2(w,f(x_i)).$$

Assume that $f(x_2) \notin [f(x_1),f(x_3)]$. Take a point $w \in f(W_{x_1,x_2,x_3})$ such that $f(x_2) \in [f(x_1),w]\cap [f(x_3),w]$, and $||w-f(x_2)||_p \in \N$, see Fig. \ref{fig:order_on_hor_line}. Then $\rho_2(w,f(x_2)) = ||w-f(x_2)||_p < \min(\rho_2(w,f(x_1)), \rho_2(w,f(x_3))).$ This contradicts $(*')$.

Hence, $f(x_2)\in [f(x_1),f(x_3)]$.
\end{proof}

\begin{figure}[ht]
\center{\includegraphics[scale=0.5, width=250pt]{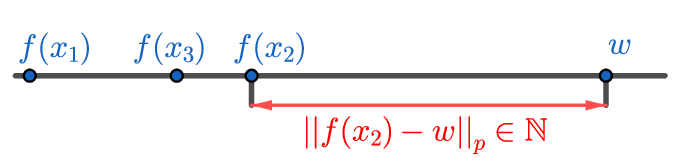}}
\caption{The point $w \in f(W_{x_1,x_2,x_3})$ such that $f(x_2) \in [f(x_1),w]\cap [f(x_3),w]$, and $||w-f(x_2)||_p \in \N$.}
\label{fig:order_on_hor_line}
\end{figure}

The following lemma shows that $f$ preserves the distances on the horizontal lines.

\begin{lemma}\label{l:ratio}
    Let $x,y \in L_1$ be on the same horizontal line, and let $||x-y|| \in \Q$. Then $||x-y||_p$ = $||f(x)-f(y)||_p$.
\end{lemma}

\begin{proof}
Due to Proposition \ref{prop:dist2}, $f$ preserves the integer distances on the horizontal lines.

    Let $x,z$ be on the same horizontal line $l$, let $y\in [x,z]$, and let $\frac{||x-y||_p}{||z-y||_p} = \frac{r}{q}$. Take a horizontal plane $P$ such that $l \in P$ (it exists because $n\geq 2$.
    Then, by elementary geometric reasoning, there exist different horizontal lines $l_x,l_y,l_z \subset P$ and different points $x_0=x,x_1,...,x_r, z_0=z,z_1,...,z_q \in P$, such that:

    $(i)$ $l\notin \{l_x,l_y,l_z\}$;

    $(ii)$ $\forall i \in [r]:\; ||x_i - x_{i+1}||_p = 1$;

    $(iii)$ $\forall i \in [q]:\; ||z_i - z_{i+1}||_p = 1$;

    $(iv)$ $x_r, z_q, y \in l_y$;

    $(iv)$ $y \in [x_r,z_q]$;

    $(vi)$ $l_x \cap l_z = \varnothing$,

    as is shown in Fig. \ref{fig:ratio}. 
    
    Then, due to Lemma \ref{l:hor_line}, we get the analogous properties $(i)-(vi)$ for $f(l_x),f(l_y),f(l_z)$, $f(l),f(x_i),f(z_i)$. Hence, $\frac{||f(x)-f(y)||_p}{||f(z)-f(y)||_p} = \frac{r}{q} = \frac{||x-y||_p}{||z-y||_p}$.

    Then $f$ preserves the rational distances on the horizontal lines (that is, if $||x-y||_p \in \Q$, then $||f(x)-f(y)||_p \in \Q$). Due to Lemma \ref{l:hor_line}, $f$ preserves the order of points on the horizontal lines. Therefore, $f$ preserves all the real distances on the horizontal lines.
\end{proof}

\begin{figure}[ht]
\center{\includegraphics[scale=0.5, width=250pt]{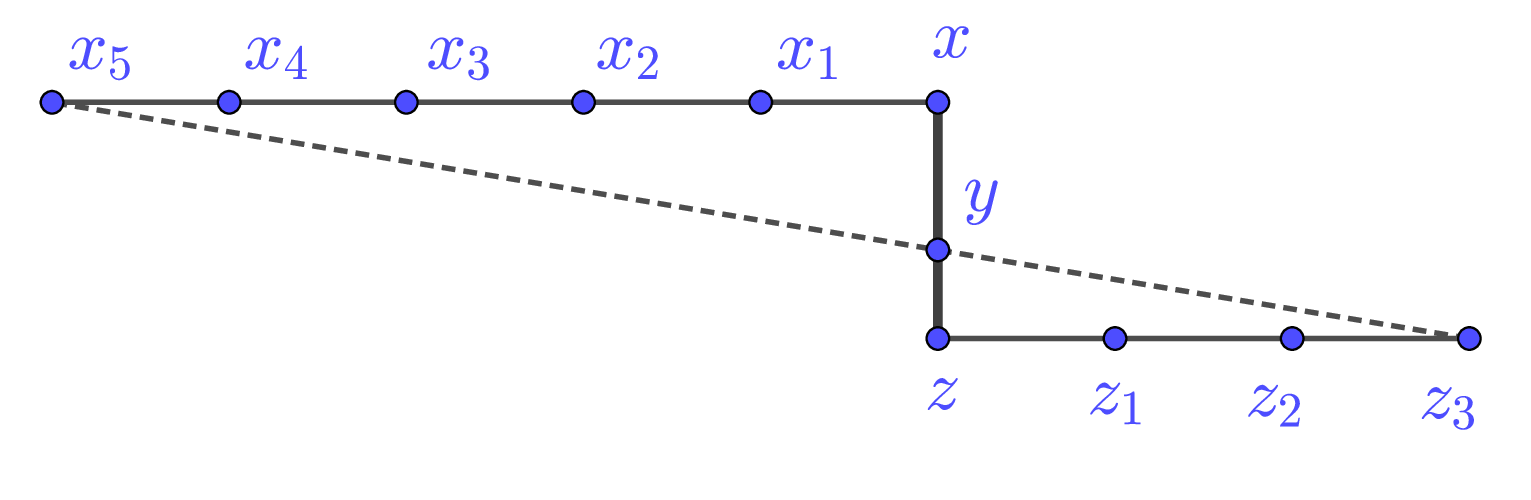}}
\caption{The points $x,y,z$ are on the same horizontal line, $\frac{||x-y||_p}{||z-y||_p} = \frac{5}{3}$}.
\label{fig:ratio}
\end{figure}

Now we can introduce a simple construction that distinguishes $\eps_1\neq\eps_2$, instead of $(N,M)$-comb in the proof of Theorem \ref{t:planar}.

Let $d_{m,p}$ be the diameter of the metric space on $[0,1]^m$ with $l_p$-metric.
\begin{lemma}\label{l:final}
    The condition $\eps_i \geq \eps$ holds if and only if there exist points $x,y,z \in L_i$ and numbers $k \in \N, \delta \in \R$ such that $y \in V_x, z\in H_x, ||y-z||_p = k, ||z-x||_p = \delta$, and $k^p - \delta^p \geq \eps^p\cdot d_{m,p}^p$.
\end{lemma}

\begin{proof}
The diameter of the set $H_x$ in $L_i$ is equal to $\eps\cdot d_{m,p}$.
    Since $\T = \H \times \V$, it is enough to apply the equality
    $$||y-z||_p = \left(||x-y||_p^p+||x-z||_p^p \right)^{1/p}.$$
\end{proof}

\begin{corollary}\label{cor:final}
    $\eps_1 \geq \eps$ if and only if $\eps_2 \geq \eps$.
\end{corollary}

\begin{proof}
    By Lemma \ref{l:final}, if $\eps_1 \geq \eps$, then there exist points $x,y,z \in L_1$ and numbers $k \in \N, \delta \in \R$ such that $y \in V_x, z\in H_x, ||y-z||_p \leq k, ||z-x||_p = \delta$, and $k^p - \delta^p \geq \eps^p\cdot d_{m,p}^p$. 
    
    Due to Lemma \ref{l:hor_line}, the line $f(x)f(z)$ is horizontal because $xz$ is horizontal. Then $f(z) \in H_{f(x)}$. Due to Lemma \ref{l:vert2}, $f(y) \in V_{f(x)}$, because $y \in V_x$. Due to Lemma \ref{l:ratio}, we have $||f(z)-f(x)||_p = ||z-x||_p = \delta$. Due to Proposition \ref{cor:distk}, we have $||f(y)-f(z)|| = ||y-z||=k$.

    So, there exist points $f(x),f(y),f(z) \in L_2$ and numbers $k \in \N, \delta \in \R$ such that $f(y) \in V_{f(x)}, f(z)\in H_{f(x)}, ||f(y)-f(z)||_p \leq k, ||f(z)-f(x)||_p = \delta$, and $k^p - \delta^p \geq \eps^p\cdot d_{m,p}^p$. Then, due to Lemma \ref{l:final}, $\eps_2 \geq \eps$.
\end{proof}

Hence, by Corollary \ref{cor:final}, we get $\eps_1=\eps_2$, which completes the proof of Theorem \ref{t:main}.

\section{Appendix: proof of Theorem \ref{t:appendix}}

Theorem \ref{t:appendix} is a simple corollary of the proof of Theorem \ref{t:main}.

Let $\phi$ be an automorphism of the unit distance graph $G$ of layer $L = L(n,1,2,\eps) = \R^n \times [0,\eps]$.

    Due to Lemmas \ref{l:vert2}, \ref{l:hor_line}, Lemma \ref{l:ratio} applied to $L_1=L_2 = L, f=\phi$, the map $\phi$ preserves the horizontal lines, the vertical segments, and the distances on the horizontal lines.

    Let $x,y \in L$ form a vertical segment and let $||x-y||_2=\alpha$. There exist $z \in L$ and $k \in \N$ such that $z \in H_x, ||z-x||_2 = \sqrt{k^2-\alpha^2}, ||z-y||_2=k$. Then $\phi(z) \in H_{\phi(x)}$, $\phi(y) \in V_{\phi(x)}$, $||z-y||_2 = ||\phi(z)-\phi(y)||_2 = k$, and $||\phi(z)-\phi(x)||_2 = ||z-x||_2 = \sqrt{k^2-\alpha^2}$. Hence,    
    $$||\phi(x)-\phi(y)||_2 = \sqrt{||\phi(z)-\phi(y)||_2^2 - ||\phi(z)-\phi(x)||_2^2} = \alpha.$$

     Thus, $\phi$ preserves the lengths of the vertical segments and the horizontal segments. Since $||\cdot||_2$ is the standard Euclidean norm, $\phi$ is an isometry.

%\section*{Statements and Declarations}
%The authors declare that no funds, grants, or other support were received during the preparation of this manuscript.
\end{document}